\documentclass[12pt]{article}

\usepackage[margin=1in]{geometry} 
\usepackage[utf8]{inputenc} 
\usepackage{amsmath}        
\usepackage{amsthm}			
\usepackage{amssymb}        
\usepackage{enumitem}       
\usepackage{url}            
\usepackage[usenames,dvipsnames]{xcolor}
\usepackage{algpseudocode}
\usepackage{algorithm}
\usepackage{mathtools}
\usepackage{graphicx}
\usepackage{caption}
\usepackage{subcaption}
\usepackage[hidelinks]{hyperref} 

\newtheorem{problem}{Problem}
\newtheorem{theorem}{Theorem}
\newtheorem{proposition}[theorem]{Proposition}

\newtheorem{definition}[theorem]{Definition}

\newcommand{\II}{\mathrm{I\!I}}
\newcommand{\mcal}{\mathcal{M}}
\newcommand{\ucal}{\mathcal{U}}

\newcommand{\ecal}{\mathcal{E}}

\newcommand{\wcal}{\mathcal{W}}
\newcommand{\Rbb}{\mathbb{R}}

\newcommand{\pa}[1]{\left(#1\right)}

\newcommand{\pabigg}[1]{\bigg(#1\bigg)}

\newcommand{\pac}[1]{\left[#1\right]}

\newcommand{\paa}[1]{\left\{#1 \right\}}

\newcommand{\norm}[1]{\left\lVert#1\right\rVert}
\newcommand{\normSmall}[1]{\lVert #1 \rVert}
\newcommand{\normbig}[1]{\big\lVert#1 \big\rVert}

\newcommand{\rest}[1]{{\big\vert}_{#1}}

\newcommand{\Retr}{\mathrm{R}}
\newcommand{\scalp}[2]{\left\langle #1,#2 \right\rangle}
\newcommand{\map}[5]{\begin{aligned}#1 \,:\, #2 &\to #3 \\ #4 &\mapsto #5\end{aligned}}

\newcommand{\dd}{\mathrm{d}}
\newcommand{\D}{\mathrm{D}}

\newcommand{\ddt}[2]{\frac{\dd #1}{\dd #2}}

\newcommand{\Exp}{\mathrm{Exp}}

\newcommand{\PiOnM}{\Pi_\mcal}

\newcommand{\PiOnMrOf}[1]{\Pi_{\mcal_r}\hspace*{-3pt}\pa{#1}}

\newcommand{\matr}[2]{\Rbb^{#1 \times #2}}

\newcommand{\minOn}[1]{\underset{#1}{\min}}
\newcommand{\maxOn}[1]{\underset{#1}{\max}}
\newcommand{\underbraceWith}[2]{\underset{#1}{\underbrace{#2}}}

\newcommand{\Rmn}{\matr{m}{n}}
\newcommand{\Rnn}{\matr{n}{n}}

\newcommand{\opt}{\mathrm{opt}}
\newcommand{\tspace}[2]{T_{#1}\hspace*{-1pt}#2}
\newcommand{\Deltat}{\Delta t}

\newcommand{\Diff}[3]{\D #1\hspace*{-0.06cm}\pa{#2}\hspace*{-0.06cm}\pac{#3}}

\newboolean{ShowRevisions}  
\setboolean{ShowRevisions}{false}   
\ifthenelse{\boolean{ShowRevisions}}
{	
	\newcommand{\revision}[1]{\textcolor{blue}{#1}}
}
{
	\newcommand{\revision}[1]{#1}
}

\title{From low-rank retractions \\  to dynamical low-rank approximation \\ and back} 
\author{Axel Séguin\thanks{École Polytechnique Fédérale de Lausanne (EPFL) Institute of Mathematics, CH-1015 Lausanne, Switzerland (axel.seguin@epfl.ch, daniel.kressner@epfl.ch)} \and Gianluca Ceruti\thanks{University of Innsbruck, Numerical Analysis and Scientific Computing, Innsbruck, Austria, gianluca.ceruti@uibk.ac.at. Parts of this work were performed while this author was part of the EPFL Institute of Mathematics, and it was supported by the SNSF research project \emph{Fast algorithms from low-rank updates}, grant number: 200020\_178806.} \and Daniel Kressner\footnotemark[1]}
\date{}

\begin{document}
	
	\maketitle
	
\begin{abstract}
In algorithms for solving optimization problems constrained to a smooth manifold, retractions are a well-established tool to ensure that the iterates stay on the manifold. More recently, it has been demonstrated that retractions are a useful concept for other computational tasks on manifold as well, including interpolation tasks. In this work, we consider the application of retractions to the numerical integration of differential equations on fixed-rank matrix manifolds. This is closely related to dynamical low-rank approximation (DLRA) techniques. In fact, any retraction leads to a numerical integrator and, vice versa, certain DLRA techniques bear a direct relation with retractions.
As an example for the latter, we introduce a new retraction, called KLS retraction, that is derived from the so-called unconventional integrator for DLRA. We also illustrate how retractions can be used to recover known DLRA techniques and to design new ones. In particular, this work introduces two novel numerical integration schemes that apply to differential equations on general manifolds: the accelerated forward Euler (AFE) method and the \revision{Projected Ralston--Hermite} (\revision{PRH}) method. Both methods build on retractions by using them as a tool for approximating curves on manifolds. The two methods are proven to have local truncation error of order three. Numerical experiments on classical DLRA examples highlight the advantages and shortcomings of these new methods.	
\end{abstract}	
	

	\section{Introduction}	\label{s:intro} 

	This work is concerned with the numerical integration of manifold-constrained ordinary differential equations (ODEs) and, in particular, the class of ODEs evolving on fixed-rank matrix manifolds encountered in dynamical low-rank approximation (DLRA)~\cite{kochLubich}. This model-order reduction technique has attracted significant attention in recent years, finding numerous applications in fundamental sciences and engineering~\cite{jahnke08,einkemmer18,peng20,kusch22}. A key feature of DLRA is that it can bring down the computational cost, in terms of storage and computational time, without compromising accuracy when solving certain large-scale \revision{ODEs} as they arise, for instance, from the discretization of partial differential equations. 
	
	Several works have already pointed out the close connection of specific DLRA techniques to retractions on fixed-rank manifolds, the latter featuring prominently in Riemannian optimization algorithms~\cite{absilBook,boumalBook}. For instance, the projector-splitting integrator for DLRA proposed in~\cite{lubichOseledets} is shown in~\cite{absilOseledets} to define a second-order retraction. The other way around, the more frequently encountered fixed-rank retraction defined with the metric projection on the manifold is centrally used in the projected Runge--Kutta methods for DLRA developed in~\cite{kieriVandereycken}. Furthermore, the recently proposed dynamically orthogonal Runge--Kutta~\cite{lermPerturbativeRetr} schemes extend Runge--Kutta schemes to the DLRA setting by making use of high-order approximations to the metric projection retraction. 
	In this work, we consider a unified framework for relating popular DLRA time integration techniques to particular retractions on the fixed-rank manifold. This not only covers the existing relations mentioned above but it also allows us to provide a novel geometric interpretation of the so-called unconventional integrator~\cite{unconventionalIntegrator} by showing that it defines a second-order retraction that coincides with the orthographic retraction \revision{modulo a third-order correction} term.
	
    Low-rank retractions can also be used to design new time integration algorithms on fixed-rank matrix manifolds, based on the insight~\cite{myThesis} that retractions can be used to conveniently generate manifold curves. For instance, any retraction allows one to define a manifold curve passing through a prescribed point with prescribed velocity. If the retraction additionally possesses the so-called second-order property, the acceleration can also be prescribed. In addition, as put forth in~\cite{mySecondPaper}, a retraction can be used to define a smooth manifold curve with prescribed endpoints and prescribed endpoint velocities. We use these two retraction-based curves as the building block for two novel numerical integration schemes for manifold-constrained ODEs and in particular for DLRA, namely the accelerated forward Euler method and the \revision{Projected Ralston--Hermite (PRH)} method. 
 
    \revision{The proposed novel numerical integrators are then tested in two classical scenarios: the Lyapunov differential equation and a synthetic example~\cite[\S 2.1]{kieriLubichWallach}. The first scenario serves as an effective numerical test to examine the impact of model errors, i.e., the magnitude of the normal component of the projected vector field, while the second analyzes the proposed method's robustness in dealing with small singular values, a notable concern in the context of numerical integrators for DLRA. We showcase that the AFE numerical method maintains its robustness when dealing with small singular values but is sensitive to the presence of significant model errors. In contrast, the PRH method is more susceptible to the influence of small singular values but maintains accuracy levels comparable to the second order Projected Runge--Kutta (PRK) scheme~\cite{kieriVandereycken}. Furthermore, for large model errors, it is illustrated that the computational time of the proposed PRH method is comparable to that of the PRK counterpart. The computational time of the accelerated methods is primarily dominated by the computation of the acceleration, namely, the Weingarten map.}

	\paragraph{Outline}  Section~\revision{\ref{s:prelims}} collects preliminary material on differential geometry and DLRA. The link between retractions and existing DLRA techniques is discussed in Section~\ref{s:retractionsAndDLRA}. Then, Section~\ref{s:newRetractionBasedAlgorithms} presents two new retraction-based integration schemes: the Accelerated Forward Euler method in Section~\ref{ss:afe} and the \revision{Projected Ralston--Hermite} integration scheme in Section~\ref{ss:rh}. 
	Numerical experiments reported in Section~\ref{s:numericalExperiments} assess the accuracy and stability to small singular values of both methods and provide performance comparison with state-of-the-art techniques. \revision{In Section~\ref{s:conclusion}, a} discussion on future work directions concludes the manuscript.

	\section{Preliminaries} \label{s:prelims}
	
	\subsection{Differential geometry} \label{s:background}

	In this section, we briefly introduce concepts from differential geometry needed in this work. The reader is referred to~\cite{boumalBook} for details.

	In the following, $\mcal$ denotes a manifold embedded into a finite-dimensional Euclidean space $\ecal$. When necessary, $\mcal$ is endowed with the Riemannian submanifold geometry: for every $x\in\mcal$, the tangent space $T_x\mcal$ is equipped with the induced metric inherited from the inner product of $\ecal$. The orthogonal complement of $T_x\mcal$ in $\ecal$ is the normal space denoted by $N_x\mcal$. The orthogonal projections from $\ecal$ to $T_x\mcal$ and $N_x\mcal$ are respectively denoted by $\Pi(x)$ and $\Pi^{\perp}(x) := (I - \Pi)(x)$. 	
	The disjoint union of all tangent spaces is called the tangent bundle and denoted by $T \mathcal{M}$. Every ambient space point $p\in\ecal$ that is sufficiently close to the manifold admits a unique metric projection onto the manifold~\cite[Lemma 3.1]{projLikeRetractions} and we denote it by $\PiOnM(p) := {\arg\min}_{x\in\mcal}\|p-x\|$.

	For a smooth manifold curve $\gamma:\Rbb\to\mcal$, its tangent vector or velocity vector \revision{at} $t\in\Rbb$ is denoted by $\dot \gamma(t)\in T_{\gamma(t)}\mcal$. If the curve $\gamma$ is interpreted as a curve in the ambient space $\ecal$, or if $\gamma$ maps to a vector space, we use the symbol $\gamma'(t)$ to indicate its tangent vector \revision{at $t\in\Rbb$}. 
	
	The Riemannian submanifold geometry of $\mcal$ is complemented with the Levi-Civita connection $\nabla$ which determines the intrinsic acceleration of a smooth manifold curve $\gamma$ as $\ddot \gamma(t) := \nabla_{\dot\gamma(t)}\dot\gamma(t)$, $\forall\,t\in\Rbb$. We reserve the symbol $\gamma''(t)$ to indicate the extrinsic acceleration of the curve intended as a curve in the ambient space $\ecal$.
	
	\paragraph{Manifold of rank-$r$ matrices.} The leading example in this work is the case where $\ecal = \matr{m}{n}$ and ${\mcal = \mcal_r}$, the manifold of rank-$r$ matrices.
	We represent rank-$r$ matrices and tangent vectors to $\mcal_r$ with the conventions used in~\cite{bart} as well as in the MATLAB library Manopt~\cite{manoptPaper} that was used to perform the numerical experiments. Any $Y\in\mcal_r$ is represented in the compact factored form $Y = U\Sigma V^\top\in\mcal_r$ with $U\in\matr{m}{r}$, ${V\in\matr{n}{r}}$ both column orthonormal and $\Sigma = \mathrm{diag}(\sigma_1,\dots,\sigma_r)$ containing the singular values ordered as ${\sigma_1\geq\dots\geq\sigma_r}$. Any tangent vector $T\in T_Y\mcal_r$ is represented as $T = UMV^\top + U_{\mathrm{p}}V^\top + UV_{\mathrm{p}}^\top$, where $M\in\matr{r}{r}$, $U_{\mathrm{p}}\in\matr{m}{r}$ and $V_{\mathrm{p}}\in\matr{n}{r}$ such that $U^\top U_{\mathrm{p}} = V^\top V_{\mathrm{p}} = 0$.  
	The manifold $\mcal_r$ inherits the Euclidean metric from $\matr{m}{n}$ so each tangent space is endowed with the inner product ${\scalp{W}{T} = \mathrm{Tr}(W^\top T)}$ and the Frobenius norm $\norm{W} = \norm{W}_{\mathrm{F}}$, for any $W,T\in T_Y\mcal_r$. Then, the orthogonal projection onto the tangent space at $Y$ of any $Z\in\Rmn$ is given by
	\begin{equation}
		\label{eq:tangentSpaceProjectionFixedRank}
		\Pi(Y) Z = UU^\top ZVV^\top +  (I-UU^\top) ZVV^\top +  UU^\top Z(I-VV^\top).
	\end{equation} 
	The metric projection onto the manifold $\mcal_r$ of a given $A\in\Rmn$ is uniquely defined when $\sigma_r\pa{A} > \sigma_{r+1}(A)$. It is computed as the rank-$r$ truncated SVD of $A$ and we denote it by $\PiOnMrOf{A}$.
	
	\paragraph{Retractions.}
	A retraction is a smooth map $\Retr:T\mcal \to \mcal:(x,v)\mapsto \Retr_x(v)$ defined in the neighborhood of the origin $x$ of each tangent space $T_x\mcal$ such that it holds that $\Retr_x(0) = x$ and $\D \Retr_x(0)[v] = v$, for all $v\in T_x\mcal$. 
	The defining properties of a retraction can be condensed into the fact that for any $x\in\mcal$ and ${v\in T_x\mcal}$ the map $\tau\in\Rbb\mapsto\sigma_{x,v}(\tau) = \Retr_x(\tau v)$ is well-defined for any sufficiently small $\tau$ and parametrizes a smooth manifold curve that satisfies $\sigma_{x,v}(0) = x$, $\dot\sigma_{x,v}(0) = v$. If the manifold is endowed with a Riemannian structure and it holds that $\ddot\sigma_{x,v}(0) = 0$ for any $x$ and $v$, then the retraction is said to be of second-order. 
	
\subsection{DLRA}\label{s:backgroundDLRA}

To explain the basic idea of DLRA, let us consider an initial value problem governed by a smooth vector field $F:\Rmn\to\Rmn$:
\begin{equation}\label{eq:ambientEquation}
	\begin{cases*}
		A^\prime(t) = F(A(t)), \quad t\in\pac{0,T},\\
		A(0) = A_0\in\Rmn.
	\end{cases*}
\end{equation}  
DLRA aims at finding an approximation of the so-called \emph{ambient solution} $A(t)\in \Rmn$ on the manifold of rank-$r$ matrices, for fixed $r\ll \min\paa{m,n}$, in order to improve computational efficiency while maintaining satisfactory accuracy. Indeed, representing the approximation of $A(t)$ in factored form drastically reduces storage complexity.
The central challenge of DLRA consists in computing efficiently a factored low-rank approximation of the ambient solution without having to first estimate the ambient solution and then to truncate it to an accurate rank-$r$ approximation. In the following, the rank is fixed a priori and does not change throughout the integration interval. Then, the DLRA problem under consideration associated to the ambient equation~\eqref{eq:ambientEquation} can be formalized as follows.
\begin{problem}\label{problem:projectedDynamicsProblem}
	Given a smooth vector field $F:\Rmn \to\Rmn$, an initial matrix ${A_0\in\Rmn}$ and a target rank $r$ such that $\sigma_r(A_0)>\sigma_{r+1}(A_0)$, the DLRA problem consists in determining $t\mapsto Y(t)\in\mcal_r$ solving the following initial value problem 
	\begin{equation}\label{eq:odeDLRA}
		\begin{cases}
			\revision{\dot Y(t) = \Pi(Y(t))F(Y(t))},\quad t \in\pac{0,T},\\
			Y(0) = Y_0,
		\end{cases}
	\end{equation}
	where $Y_0 =\Pi_{\mcal_r}\pa{A_0}\in\mcal_r$, the rank-$r$ truncated singular value decomposition of $A_0$.
\end{problem}
The origins of the above problem are rooted in the Dirac--Frenkel variational principle, by which the dynamics of~\eqref{eq:ambientEquation} are optimally projected onto the tangent space of the manifold\revision{. In fact, for any $Y\in\mcal_r$, the orthogonal projection of $F(Y)$ to $T_Y\mcal_r$ returns the closest vector in the tangent space:}
\begin{equation}\label{eq:modelingError}				
	\norm{F(Y) - 	\Pi(Y)F(Y)}_{\mathrm{F}} = \minOn{V\in \tspace{Y}{\mcal_r}} \norm{F(Y) - V}_{\mathrm{F}}.
\end{equation}	
This optimality criterion together with the optimal choice of $Y_0$ are the local first-order approximation of the computationally demanding optimality $Y_{\opt}(t) = \Pi_{\mcal_r}\pa{A(t)}$. The gap~\eqref{eq:modelingError} between the original dynamics and the projected dynamics of Problem~\ref{problem:projectedDynamicsProblem} is known as the \emph{modeling error}~\cite{kieriVandereycken}. 
Given appropriate smoothness requirements on $F$ and assuming the modeling error can be uniformly bounded in a neighborhood \revision{$\ucal\subset \matr{m}{n}$ of the trajectory $Y([0,T])\subset\mcal$} of the exact solution of~\eqref{eq:odeDLRA} as
\begin{equation}
	\maxOn{Y\in\,\ucal\cap\mcal_r}\norm{F(Y) - \Pi(Y)F(Y)}_{\mathrm{F}} \leq \varepsilon,
\end{equation}
then it can be shown, see e.g.~\cite[Theorem 2]{kieriVandereycken}, that there exists a constant $C>0$ depending on the final time $T$ such that 
\begin{equation}
	\norm{A(T) - Y(T)}_{\mathrm{F}} \leq C \pa{\delta_0 + \varepsilon},
\end{equation}
\revision{where $\delta_0 = \norm{A(0) - Y(0)}_{\mathrm{F}}$ denotes the initial approximation error.}

In recent years, several computationally efficient numerical integration schemes to approximate the solution to Problem~\ref{problem:projectedDynamicsProblem} have been proposed~\cite{lubichOseledets,unconventionalIntegrator,kieriVandereycken}. Their output is a time discretization of the solution, where $Y_k\in\mcal_r$ approximates  $Y(k\Deltat)$, for every ${k = 0,\dots,N}$, assuming -- for simplicity -- that a fixed step size ${\Deltat = T / N}$ is used. Existing error analysis results~\cite{kieriLubichWallach,kieriVandereycken} state that, provided the step size is chosen small enough, the error at the final time can be bounded as
\begin{equation}
	\norm{Y_N - A(T)}_{\mathrm{F}} \leq \tilde C \pa{\delta_0 + \varepsilon + \Deltat^q}, 
	\label{eq:errorConvergenceForDLRA}
\end{equation} 
for some integer $q\geq 1$ and a constant $\tilde C > 0$ 
that depends on the final time and the problem at hand. The constant $q$ is called the \emph{convergence order} of the time stepping method.

The most direct strategy to numerically integrate Problem~\ref{problem:projectedDynamicsProblem} is \revision{to} write $Y$ in rank-$r$ factorization and derive individual evolution equations for the factors~\cite{kochLubich}. However, the computational advantages of such an approach are undermined by the high stiffness of the resulting equations when the $r$th singular value of the approximation becomes small, which is often the case in applications. In turn, this enforces unreasonably small step sizes in order to guarantee the stability of explicit integrators. The projector-splitting schemes proposed by Lubich and Oseledets~\cite{lubichOseledets} were the first to remedy this issue. Since then, a collection of methods have been designed that achieve stability independently of the  presence of small singular values~\cite{kieriLubichWallach,unconventionalIntegrator,ceruti2023parallel,lermPerturbativeRetr,kazashiNobileVidlickova}. As we show in the following section, some of these DLRA algorithms are directly related to retractions.

\section{DLRA algorithms and low-rank retractions}\label{s:retractionsAndDLRA}		
The concept of retraction was initially formulated in the context of numerical time integration on manifolds. In the work of Shub~\cite{shub86}, what is now called a retraction was proposed as a generic tool to develop extensions of the forward Euler method that incorporate manifold constraints. Indeed, any retraction admits the following first-order approximation property. Let $t\mapsto Y(t)$ denote the exact solution to a manifold-constrained ODE such as~\eqref{eq:odeDLRA} but on a general embedded manifold $\mcal$. Provided the vector field governing the ODE is sufficiently smooth, the defining properties of a retraction imply that~\cite[Theorem 1]{shub86} \begin{equation}\label{eq:localTruncationErrorRetraction}
	\normSmall{Y(t+\tau)- \Retr_{Y(t)}(\tau \dot Y(t))} = O(\tau^2).
\end{equation}
This mirrors the local truncation error achieved by the forward Euler method on a Euclidean space while ensuring that the curve $\tau \mapsto \Retr_{Y(t)}(\tau \dot Y(t))$ remains on the constraint manifold. Hence, if $Y_{k}$ indicates an approximation to the exact solution $Y(t_k)$  at time $t = t_k$, the update rule	\begin{equation}\label{eq:prototypeIntegrationScheme}
	Y_{k+1} = \Retr_{Y_k}\pa{\Deltat \dot Y_{k}}
\end{equation}	
for approximating the solution at $t = t_{k+1} = t_k + \Deltat$, for a given step size $\Deltat>0$,
is a retraction-based generalization of the forward Euler for manifold-constrained ODEs.
In virtue of~\eqref{eq:localTruncationErrorRetraction}, for any retraction on $\mcal$, the scheme~\eqref{eq:prototypeIntegrationScheme} achieves local truncation error of order two  \revision{(see Section~\ref{s:newRetractionBasedAlgorithms} below for the definition of the order), which implies global first-order convergence under certain conditions~\cite[Section II.3]{hairerSolveODEs1}.} 

Let us now specialize \revision{the discussion on the relation between retractions and numerical integration of manifold-constrained ODEs} to the DLRA differential equation of Problem~\ref{problem:projectedDynamicsProblem} evolving on the fixed-rank manifold $\mcal_r$. Implementing the general scheme~\eqref{eq:prototypeIntegrationScheme} requires the choice of a particular low-rank retraction.	A substantial number of retractions are known for the fixed-rank manifold~\cite{absilOseledets} and, as discussed in the following, several existing DLRA integration schemes are realized as~\eqref{eq:prototypeIntegrationScheme} for a \revision{particular} choice of retraction on $\mcal_r$. 

\subsection{Metric-projection retraction} 
 A large class of methods to numerically integrate manifold-constrained ODEs fit into the class of \emph{projection methods}; see~\cite[\S IV.4]{hairerYellowBook}. Each integration step is carried out in the embedding space with a Euclidean time-stepping method and is followed by the metric projection onto the constraint manifold. For the DLRA differential equation of Problem~\ref{problem:projectedDynamicsProblem}, projection methods have been studied in~\cite{kieriVandereycken}. One step of the \emph{projected forward Euler method} takes the form
 \begin{equation}\label{eq:projectedForwardEuler}
 	Y_{k+1} = \Pi_{\mcal_r}(Y_{k} + \Deltat \Pi(Y_k) F(Y_k)),
 \end{equation} 
 where $\Pi_{\mcal_r}$ coincides with the rank-$r$ truncated SVD. 
 This integration scheme fits into the general class of retraction-based schemes~\eqref{eq:prototypeIntegrationScheme}. Indeed, given $X\in\mcal_r$ and $Z\in T_X\mcal_r$, the metric projection of the ambient space point $X+Z$ defines a retraction~\cite[\S 3]{projLikeRetractions},
 \begin{equation}\label{eq:retractionSVD}
 	\Retr_X^{\mathrm{SVD}}(Z):= \PiOnMrOf{X+Z},
 \end{equation}
which we refer to as the SVD retraction.
Since $Z\in T_X\mcal_r$, the matrix $X + Z$ is of rank at most $2r$. This allows for an efficient implementation of this retraction, as detailed in Algorithm~\ref{alg:retractionSVD}. 

The projected forward Euler method has order $q=1$~\cite[Theorem 4]{kieriVandereycken}. To achieve higher orders of convergence, \emph{projected Runge--Kutta} (PRK) methods have been proposed~\cite[\S5]{kieriVandereycken}. The intermediate stages of such methods are obtained by replacing the forward Euler update in~\eqref{eq:projectedForwardEuler} with the updates of an explicit Runge--Kutta \revision{method}. Although the update vectors may not belong to a single tangent space, the PRK recursion can still be written because the SVD retraction has the particular property that it remains well-defined not only for tangent vectors but also for sufficiently small general vectors $Z\in\Rmn$ (a property that makes it an \emph{extended retraction}~\cite[\S 2.3]{absilOseledets}).
We shall abbreviate the PRK method with $s\geq 1$ stages as PRK$s$. Note that PRK1 coincides with the projected forward Euler method above.
	
	\begin{algorithm}
		\caption{SVD retraction on $\mcal_r$.}\label{alg:retractionSVD}		
		\textbf{Input:} $X = U_0\Sigma_0 V_0^\top\in\mcal_r$, $Z =  UMV^\top + U_{\mathrm{p}}V^\top + UV_{\mathrm{p}}^\top\in T_X\mcal_r$.
		\begin{algorithmic}[1]
			\State $Q_U R_U = U_{\mathrm{p}}$ with $Q_U$ orthonormal;
			\State $Q_V R_V = V_{\mathrm{p}}$ with $Q_V$ orthonormal;
			\State $[U_{\mathrm{S}},\Sigma_{\mathrm{S}},V_{\mathrm{S}}] = \mathrm{SVD}\pa{\begin{bmatrix}
					\Sigma + M &R_V^\top\\
					R_U        &0
			\end{bmatrix}};$
			\State $\Sigma_1 = \Sigma_{\mathrm{S}}(1:r,1:r)$;
			\State $U_1 = \begin{bmatrix}	U_0 &Q_U \end{bmatrix}U_{\mathrm{S}}(:,1:r)$;
			\State $V_1 = \begin{bmatrix}	V_0 &Q_V \end{bmatrix}V_{\mathrm{S}}(:,1:r)$;\\
			\Return : $U_1 \Sigma_1 V_1^\top \revision{=:} \Retr_X^\mathrm{SVD}(\revision{Z})$;
		\end{algorithmic}
	\end{algorithm}
 
\subsection{Projector-splitting KSL retraction}
 
 The expression~\eqref{eq:tangentSpaceProjectionFixedRank} for the orthogonal projection of $Z\in\Rmn$ onto the tangent space at $Y = U\Sigma V^\top\in\mcal_x$ can be expressed as
 \begin{equation}\label{eq:alternativeTangentSpaceProjection}
 	\Pi(Y)Z = ZVV^\top - UU^\top ZVV^\top + UU^\top Z.  
 \end{equation}
 The \emph{projector-splitting} scheme for DLRA introduced in~\cite{lubichOseledets} is derived by applying this decomposition to the right-hand side of~\eqref{eq:odeDLRA} and using standard Lie--Trotter splitting. Each integration step comprises three integration substeps identified with the letters K, S and L, corresponding respectively to the three terms in~\eqref{eq:alternativeTangentSpaceProjection}. When each of the integration substeps is performed with a forward Euler update, we refer to this scheme as the KSL scheme. The scheme is proved to be first-order accurate independently of the presence of small singular values~\cite[Theorem 2.1]{kieriLubichWallach}. 
 
 As shown in~\cite[Theorem 3.3]{absilOseledets}, one step of the KSL scheme actually defines a second-order retraction for the fixed-rank manifold. In fact, it coincides \revision{modulo third-order terms} with the orthographic retraction presented in  Section~\ref{ss:klsRetractionAndTheOrthographicRetraction}. The KSL retraction is denoted by $\Retr^{\mathrm{KSL}}$ and its computation is summarized in Algorithm~\ref{alg:retractionKSL}. Hence, the KSL integration scheme fits into the general scheme~\eqref{eq:prototypeIntegrationScheme} for  Problem~\ref{problem:projectedDynamicsProblem} as it can simply be written as 
	\begin{equation}
		Y_{k+1} = \Retr_{Y_k}^{\mathrm{KSL}}\pa{\Deltat \Pi(Y_k) F(Y_k)}.
	\end{equation}

	\begin{algorithm}[ht]
		\caption{KSL retraction}\label{alg:retractionKSL}		
		\textbf{Input:} $X = U_0\Sigma_0 V_0^\top \in\mcal_r$, $Z =  UMV^\top + U_{\mathrm{p}}V^\top + UV_{\mathrm{p}}^\top\in T_X\mcal_r$.
		\begin{algorithmic}[1]
			\State (K-step) $U_1\hat \Sigma_1 = U_0(\Sigma_0 + M) + U_{\mathrm{p}}$ with $U_1$ orthonormal;
			\State (S-step) $\tilde \Sigma_0 = \hat \Sigma_1 - (U_1^\top  U_{\mathrm{p}} + (U_1^\top  U_0)M)$;
			\State (L-step) $V_1 \Sigma_1^\top  = V_0\tilde\Sigma_0^\top  + Z^\top  U_1$ with $V_1$ orthonormal; 
			\State Optional: $[\tilde U_1, \tilde\Sigma_1, \tilde V_1] = \mathrm{SVD}(\Sigma_1)$;
			\State Optional: $U_1\leftarrow U_1 \tilde U_1$, $\Sigma_1\leftarrow \tilde \Sigma_1$, $V_1\leftarrow V_1 \tilde V_1$;\\
			\Return : $U_1 \Sigma_1 V_1^\top  \revision{=:} \Retr_X^\mathrm{KSL}(\revision{Z})$;
		\end{algorithmic}
	\end{algorithm}

\subsection{Projector-splitting KLS retraction}
	Recently, a modification to KSL projector-splitting method~\cite{unconventionalIntegrator} was proposed to improve its (parallel) performance, while maintaining the stability and accuracy properties of the KSL method. The scheme is known as the \emph{unconventional integrator} and it is a modification of the KSL scheme where the L-step is performed before the S-step. Hence,  when each of the integration substeps is performed with a forward Euler update, we refer to it as the KLS scheme. The KLS scheme comes with the computational advantage of allowing the K-step and L-step to be performed in parallel without compromising first-order accuracy and stability with respect to small singular values~\cite[Theorem 4]{unconventionalIntegrator}. 
	
	As we prove in the next section, one step of the KLS scheme also defines a retraction for the fixed-rank manifold. We denote it by $\Retr^\mathrm{KLS}$ and its computation is detailed in Algorithm~\ref{alg:retractionKLS}. 
	Then, as for other \revision{schemes} so far, the KLS integration scheme for DLRA takes the simple form 
	\begin{equation}
		\revision{Y_{k+1}} = \Retr_{Y_k}^{\mathrm{KLS}}\pa{\Deltat \Pi(Y_k) F(Y_k)}.
	\end{equation}

	\begin{algorithm}
		\caption{KLS retraction}\label{alg:retractionKLS}		
		\textbf{Input:} $X = U_0\Sigma_0 V_0^\top \in\mcal_r$, $Z =  UMV^\top + U_{\mathrm{p}}V^\top + UV_{\mathrm{p}}^\top\in T_X\mcal_r$.
		\begin{algorithmic}[1]
			\State (K-step) $U_1 S_U = U_0(\Sigma_0 + M) + U_{\mathrm{p}}$ with $U_1$ orthonormal;
			\State (L-step) $V_1 S_V = V_0(\Sigma_0^\top + M^\top ) + V_{\mathrm{p}}$ with $V_1$ orthonormal;
			\State $L = U_1^\top  U_0$;
			\State $R = V_1^\top  V_0$;
			\State (S-step) $\Sigma_1 = L\pac{\pa{\Sigma_0 + M} R^\top  + V_{\mathrm{p}}^\top  V_1} + U_1^\top  U_{\mathrm{p}} R^\top $; \Comment{equivalent to $U_1^\top  (X + Z) V_1$}			
			\State Optional: $[\tilde U_1, \tilde\Sigma_1, \tilde V_1] = \mathrm{SVD}(\Sigma_1)$;
			\State Optional: $U_1\leftarrow U_1 \tilde U_1$, $\Sigma_1\leftarrow \tilde \Sigma_1$, $V_1\leftarrow V_1 \tilde V_1$;\\
			\Return : $U_1 \Sigma_1 V_1^\top  \revision{=:} \Retr_X^\mathrm{KLS}(Z)$;
		\end{algorithmic}
	\end{algorithm}
	
	\subsubsection{The KLS retraction and the orthographic retraction}\label{ss:klsRetractionAndTheOrthographicRetraction}
	
	The goal of this section is to prove that the update rule of the KLS scheme, as detailed by Algorithm~\ref{alg:retractionKLS}, defines a second-order retraction. The strategy to reach this goal is based on the observation that the KLS update is obtained as a small perturbation of another commonly encountered retraction, known as the orthographic retraction.

	\begin{algorithm}
		\caption{Orthographic retraction}\label{alg:retractionORTHAsKLS}		
		\textbf{Input:} $X = U_0\Sigma_0 V_0^\top \in\mcal_r$, $Z =  UMV^\top + U_{\mathrm{p}}V^\top + UV_{\mathrm{p}}^\top\in T_X\mcal_r$.
		\begin{algorithmic}[1]
			\State $ U_1S_U = U_0(\Sigma_0 + M) + U_{\mathrm{p}}$ with $ U_1$ orthonormal;
			\State $ V_1S_V = V_0(\Sigma_0^\top + M^\top) + V_{\mathrm{p}}$ with $ V_1$ orthonormal;
			\State $\Sigma_1 = S_U (\Sigma_0 + M)^{-1} S_V^\top$,
			\State Optional: $[\tilde U_1, \tilde\Sigma_1, \tilde V_1] = \mathrm{SVD}(\Sigma_1)$;
			\State Optional: $U_1\leftarrow U_1 \tilde U_1$, $\Sigma_1\leftarrow \tilde \Sigma_1$, $V_1\leftarrow V_1 \tilde V_1$;\\
			\Return : $U_1 \Sigma_1 V_1^\top  \revision{=:} \Retr_X^\mathrm{ORTH}(Z)$;
		\end{algorithmic}
	\end{algorithm}
	
	\begin{figure}
		\centering
		\includegraphics[width=0.75\linewidth]{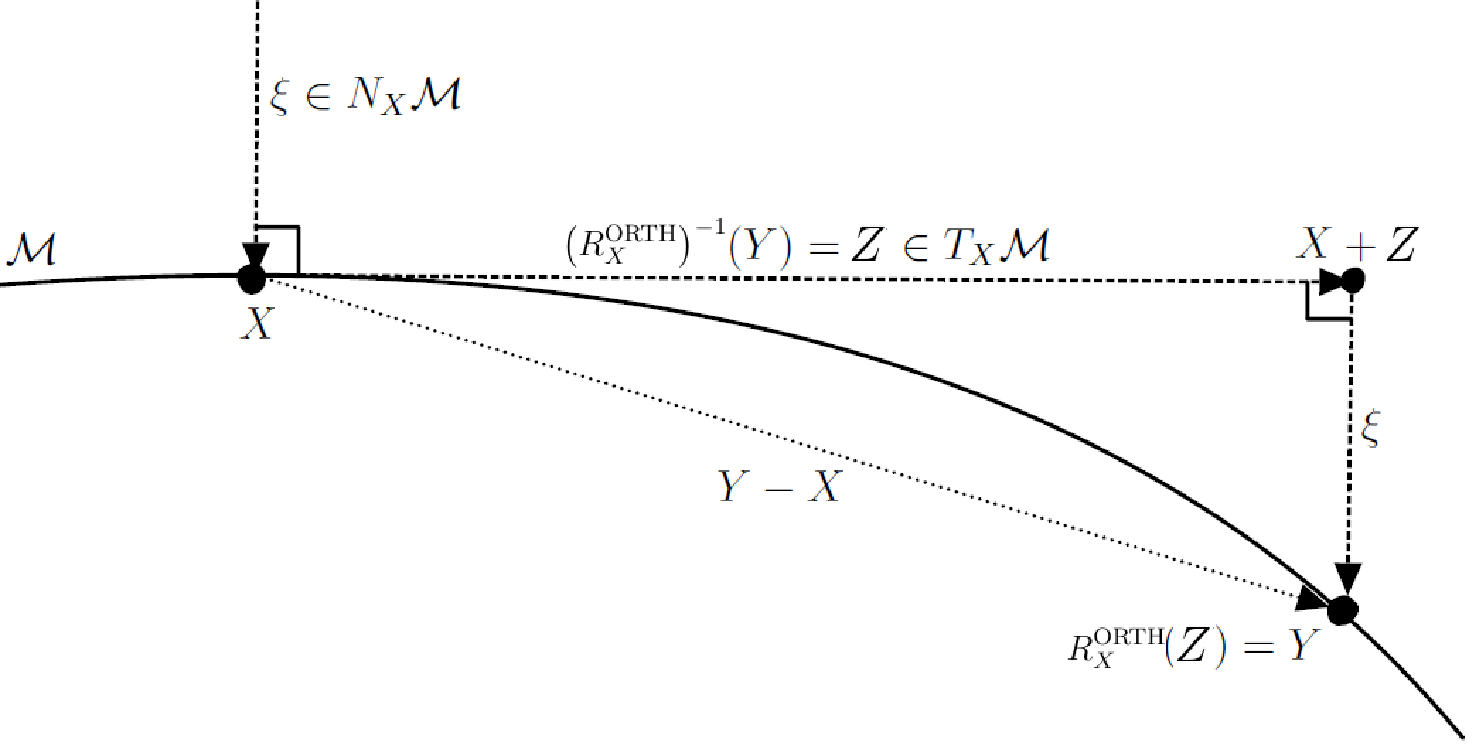}
		\caption{Orthographic retraction and its inverse.}
		\label{fig:orthographicRetr}
	\end{figure}

	\paragraph{The orthographic retraction.} 
	Specialized to the fixed-rank matrix manifold $\mcal_r$, the orthographic retraction consists in perturbing the point $X\in\mcal_r$ in the ambient space as $X+Z\in\matr{m}{n}$ and projecting back onto the manifold but only along vectors from the normal space of the starting point, see Figure~\ref{fig:orthographicRetr}. Formally this reads:
	\begin{equation}
		\Retr_X^{\mathrm{ORTH}}(Z) = \underset{Y\in(X+Z+N_X\mcal_r)\cap\mcal_r}{\arg\min}\norm{X + Z - Y}_{\mathrm{F}}.
	\end{equation}
	A closed-form expression for the solution of this optimization problem for the case of $\mcal_r$ is established in \cite[\S 3.2]{absilOseledets} and is the output of Algorithm~\ref{alg:retractionORTHAsKLS}. In virtue of the analysis carried out in~\cite{projLikeRetractions}, the orthographic retraction is a second-order retraction. A remarkable property of the orthographic retraction is that its local inverse can be computed easily. As suggested by the construction  illustrated in Figure~\ref{fig:orthographicRetr}, the inverse orthographic retraction is obtained by projecting the ambient space difference onto the tangent space, i.e.,
	\begin{equation}
		\pa{\Retr_X^{\mathrm{ORTH}}}^{-1}(Y) = \Pi(X)(Y-X).
	\end{equation}
	This opens the possibility to use the inverse orthographic retraction in practical numerical procedures such as Hermite interpolation of a manifold curve~\cite{mySecondPaper} used in the definition of the \revision{Projected Ralston--Hermite} scheme in Section~\ref{ss:rh}.

	A careful inspection of the computations for the orthographic retraction reported in Algorithm~\ref{alg:retractionORTHAsKLS} reveals that it is identical to the update rule of the KLS scheme given in Algorithm~\ref{alg:retractionKLS}, up to the computation of $\Sigma_1$. Let  $\Sigma_1^{\mathrm{KLS}}$ and $\Sigma_1^{\mathrm{ORTH}}$ denote the quantities computed respectively at step 5 of Algorithm~\ref{alg:retractionKLS} and step 3 of Algorithm~\ref{alg:retractionORTHAsKLS}. Note that $\Sigma_{1}^{\mathrm{ORTH}}$ can be rewritten without explicitly computing the factors $S_U$ and $S_V$. Using the relations 
	\begin{equation}
		S_U = L(\Sigma_0 + M) + U_1^\top U_{\mathrm{p}},\quad S_V = R(\Sigma_0^\top + M^\top) + V_1^\top  V_{\mathrm{p}},
	\end{equation}
	with $L = U_1^\top  U_0$, $R = V_1^\top  V_0$ we evince that
	\begin{align}
		\Sigma_1^{\mathrm{ORTH}} &= S_U (\Sigma_0 + M)^{-1} S_V^\top \\
		&= L \pac{(\Sigma_0 + M) R^\top  + V_{\mathrm{p}}^\top   V_1} +  U_1^\top  U_{\mathrm{p}} R^\top   +  U_1^\top  U_{\mathrm{p}} (\Sigma_0 + M)^{-1}V_{\mathrm{p}}^\top   V_1.
	\end{align}
	 This highlights that the quantity $\Sigma_1$ computed in the KLS scheme and the orthographic retraction differ by
	\begin{equation}
		\Sigma_1^{\mathrm{ORTH}} - \Sigma_1^{\mathrm{KLS}} =  U_1^\top  U_{\mathrm{p}} (\Sigma_0 + M)^{-1}V_{\mathrm{p}}^\top   V_1.
	\end{equation}

	Using this observation together with the second-order property of the orthographic retraction allows us to show that the KLS procedure defines a second-order retraction. This link with the orthographic retraction mirrors the same observation made for the closely related KSL scheme~\cite[Theorem 3.3]{absilOseledets}. 
	\begin{proposition}
		The procedure of Algorithm~\ref{alg:retractionKLS} defines a second-order retraction (called the KLS retraction).
	\end{proposition}
	\begin{proof}
	The proof relies on a necessary and sufficient condition for a retraction to be a second-order retraction stated in~\revision{\cite[Proposition 3]{projLikeRetractions}}. On a general Riemannian manifold $\mcal$, a mapping $\Retr:T\mcal\to \mcal$ is a second-order retraction if and only if for all $(x,v)\in T\mcal$ the curve $t\mapsto \Retr_x(tv)$ is well-defined for all sufficiently small $t$ and satisfies
	\begin{equation}
		\Retr_x(tv) = \Exp_x(tv) + O(t^3),
	\end{equation}
 	where $t\mapsto \Exp_x(tv)$ is a parametrization of the geodesic passing through $x$ with velocity $v$ obtained with the Riemannian exponential map at $x$~\cite[Definition 10.16]{boumalBook}. 
	
		Given $X\in\mcal_r$ and $Z\in T_X\mcal_r$, the orthonormalizations of the first two lines of the KLS scheme in Algorithm~\ref{alg:retractionKLS} are uniquely defined provided the matrices to orthonormalize have full rank. This is the case for any $Z$ sufficiently small by lower semi-continuity of the matrix rank. By smoothness of the orthonormalization process, we know that the matrix $\Sigma_1$ depends smoothly on $Z$. Since for $Z = 0$ we have  $\Sigma_1 = \Sigma_0$, assumed to be full rank, the matrix $\Sigma_1$ has full rank for sufficiently small $Z$. Hence $U_1\Sigma_1V_1^\top  = \Retr_X^\mathrm{KLS}(Z)$ is uniquely and smoothly defined for any $Z$ in a neighborhood of the origin of $T_X\mcal_r$ and belongs to $\mcal_r$. 
		
		Now, consider the curves $t\mapsto \Retr_X^{\text{ORTH}}(tZ)$ and $t\mapsto \Retr_X^{\text{KLS}}(tZ)$, well-defined for sufficiently small $t$. Since these curves share the left and right singular vectors $U_1(t)$ and $V_1(t)$, their difference is given by
		\begin{equation*}
			{\Retr_X^{\text{ORTH}}(tZ) - \Retr_X^{\text{KLS}}(tZ)} = t^2  U_1(t)^\top  U_{\mathrm{p}}(\Sigma_0 + tM)^{-1}V_{\mathrm{p}}^\top   V_1(t)= t^2 C(t),
		\end{equation*}
		where $C(t):= U_1(t)^\top  U_{\mathrm{p}}(\Sigma_0 + tM)^{-1}V_{\mathrm{p}}^\top   V_1(t)$. Let us show that $C(t) = o(t)$, i.e. ${\lim_{t\to 0} C(t) / t = 0}$. Since $U_1(0) = U_0$ and $V_1(0) = V_0$, by definition of tangent space of $\mcal_r$, we know that $U_1(0)^\top  U_{\mathrm{p}} = 0$ and  $V_{\mathrm{p}}^\top  V_1(0) = 0$. Hence $C(0) = 0$ and therefore 
		\begin{equation}
			\lim_{t\to 0}\frac{C(t)}{t} = \lim_{t\to 0}\frac{C(t) - C(0)}{t}.
		\end{equation}
		This coincides with $C'(0)$ since $C$ is smooth for small $t$. But since
		\begin{align}
			C'(0) &= U_1'(0)^\top  U_{\mathrm{p}}\Sigma_0^{-1}\underbraceWith{=0}{V_{\mathrm{p}}^\top   V_1(0)} +\underbraceWith{=0}{ U_1(0)^\top  U_{\mathrm{p}}}\ddt{}{t}(\Sigma_0 + tM)^{-1}\rest{t=0}\underbraceWith{=0}{V_{\mathrm{p}}^\top   V_1(0)}\\ 
			&\quad + \underbraceWith{=0}{ U_1(0)^\top  U_{\mathrm{p}}}\Sigma_0^{-1}V_{\mathrm{p}}^\top   V_1'(0) = 0
		\end{align}
		we can infer that $C(t) = o(t)$ and therefore that $\Retr_X^{\text{KLS}}(tZ) = \Retr_X^{\text{ORTH}}(tZ) + o(t^3)$.
		Using the result~\cite[Proposition 2.3]{projLikeRetractions} explained in the beginning of the proof and the second-order property of the orthographic retraction, we obtain that the retraction curve $t\mapsto \Retr_X^{\text{ORTH}}(tZ)$ approximates the geodesic $t\mapsto \Exp_X(tZ)$, as $ \Retr_X^{\text{ORTH}}(tZ) = \Exp_X(tZ) + O(t^3)$. Combining the two results leads to
		\begin{align}
			\Retr_X^{\text{KLS}}(tZ) = \Retr_X^{\text{ORTH}}(tZ) + o(t^3) = \gamma_{X,Z}(t) + O(t^3) + o(t^3) = \gamma_{X,Z}(t) + O(t^3).
		\end{align}
		Using once again the result from~\cite[Proposition 2.3]{projLikeRetractions}, this implies that $\Retr_X^{\text{KLS}}$ is indeed a second-order retraction. \hfill $\qed$
	\end{proof}

	\section{New retraction-based time stepping algorithms }\label{s:newRetractionBasedAlgorithms}	

	Having seen in Section~\ref{s:retractionsAndDLRA} that existing DLRA time integration algorithms can be directly related with particular low-rank retractions, we will now discuss the opposite direction, how low-rank retractions can be used to produce new integration schemes.
	
    Before proceeding to the derivation of the new methods, we first briefly describe the rationale behind these methods for the initial value problem~\eqref{eq:ambientEquation} evolving in $\Rmn$. We let $A(\Deltat)$ denote the exact solution at time $t = \Deltat$ and $\tilde A(\Deltat)$ its approximation obtained by performing one step of a given numerical integration scheme starting from $A(0)$. The scheme is said to have order $q$ if the local truncation error satisfies
	\begin{equation} 
		 \normSmall{A(\Deltat) - \tilde A(\Deltat)} = O(\Deltat^{q+1}).
	\end{equation} 
	In other words, any curve $\tilde A(\cdot )$ on $\Rmn$ that is a sufficiently good approximation of $A( \cdot )$ in the vicinity of $0$ defines a 
	suitable numerical integration scheme.	
In the two methods we propose, the curves $\tilde A$ are constructed through manifold interpolation, based on retractions. 
    
   
    By definition, see Section~\ref{s:background}, any retraction induces manifold curves with prescribed initial position and velocity. If, additionally, the retraction has second order, it is possible to prescribe an initial acceleration.
    \begin{proposition}[{\cite[Exercise 5.46]{boumalBook}, \cite[Corollary 3.2]{myThesis}}]\label{prop:retractionCurveWithAcceleration}
    Let $\Retr$ be a retraction on $\mathcal M$ and consider arbitrary $x\in\mcal$ and $v,a\in T_x\mcal$. Then the curve
    	\begin{equation}\label{eq:retractionCurveWithAcceleration}
    		\sigma(t) = \Retr_x\pabigg{tv + \frac{t^2}{2}a}
    	\end{equation}
    	is well-defined for all $t$ sufficiently small and satisfies
    	\begin{enumerate}
    		\item $\sigma(0) = x$ and $\dot\sigma(0) = v$,
    		\item $\ddot\sigma(0) = a$ if $\Retr$ is a second-order retraction.
    	\end{enumerate} 
    \end{proposition}
	
	When setting $x = \gamma(0)$, $v = \dot\gamma(0)$ and $a = \ddot\gamma(0)$ for a smooth manifold curve $\gamma$, it is shown in~\cite[Proposition 3.22]{myThesis} that the Riemannian distance between $\gamma$ and the retraction curve constructed by Proposition~\ref{prop:retractionCurveWithAcceleration} satisfies
	\begin{equation}\label{eq:approximationPowerOfRetractionCurves}
		d(\gamma(t), \sigma(t)) = O(t^3), \quad t \to 0,
	\end{equation}
	if $\Retr$ is a second-order retraction.


	\subsection{An accelerated forward Euler  scheme}\label{ss:afe}

	The first method we propose for DLRA is obtained by accelerating the projected forward Euler scheme~\eqref{eq:projectedForwardEuler} through a second-order correction.
	
	\subsubsection{Euclidean analog}
 	For illustration, let us first derive the accelerated forward Euler method for the initial value problem~\eqref{eq:ambientEquation} evolving in $\Rmn$. If the solution to~\eqref{eq:ambientEquation} is sufficiently smooth, we get from its Taylor expansion that
 	\begin{align*}
 		A(\Deltat) - \pa{A(0) + \Deltat A'(0) + \frac{\Deltat^2}{2}A''(0)} = O(\Deltat^3).
 	\end{align*}
 	Assuming we can compute $A''(0) = \Diff{F}{A_0}{F(A_0)}$, the differential of the vector field along $F(A_0)$, the Euclidean accelerated forward Euler (AFE) scheme is defined by the update
	\begin{equation}
		\label{eq:euclideanAFE}
		A^{\mathrm{AFE}}(\Deltat) = A_0 + \Deltat F(A_0) + \frac{\Deltat^2}{2} \Diff{F}{A_0}{F(A_0)}.
	\end{equation}
		
	The Euclidean AFE scheme is of order $q = 2$ since the local truncation error is by construction of order $O(\Deltat^{3})$. The scheme is conditionally stable with the same stability domain as any 2-stages explicit Runge--Kutta method of order 2~\cite[Theorem 2.2]{hairerWanner2}. That being said, this scheme requires the evaluation of the Jacobian at each step, which renders the AFE scheme non-competitive on Euclidean spaces. In fact, for Euclidean ODEs there are simple second-order methods using only the evaluation of the vector field, such as Runge--Kutta methods. 
	
	In the following, we generalize the AFE scheme for the DLRA differential equation of Problem~\ref{problem:projectedDynamicsProblem} by first showing how to compute the acceleration of the exact solution.  As we will see below,  in Proposition~\ref{prop:dlraAcceleration}, it is the sum of two terms: the tangent component of the ambient acceleration and the acceleration due to the curvature of the manifold. The latter is expressed using the \emph{Weingarten map}, a classical concept of Riemannian geometry that we briefly introduce in the next section.

	\subsubsection{The Weingarten map}\label{sss:weingartenMap}
	As before, we consider a Riemannian submanifold $\mcal$ embedded into a finite-dimensional Euclidean space $\ecal$; see~\cite[\S 8]{leeRiemManifs2} for a more general setting.
	We let $\mathfrak{X}_{\mathrm{T}}(\mcal)$ and $\mathfrak{X}_{\mathrm{N}}(\mcal)$ denote the set of smooth tangent and normal vector fields on $\mcal$, respectively.
	
	The Weingarten map is constructed from the second fundamental form. The latter measures the discrepancy between the ambient Euclidean connection and the Riemannian connection $\nabla$ on the submanifold $\mcal$. For every $W,Z\in\mathfrak{X}_{\mathrm{T}}(\mcal)$ and $x\in\mcal$, we have that $\nabla_WZ(x) = \Pi(x)\D_W Z(x)$, were $\D$ is the Euclidean connection corresponding to standard directional derivative. Hence, the vector 
	\begin{equation}
		\D_WZ(x) - \nabla_WZ(x) = \D_WZ(x) - \Pi(x)\D_WZ(x) 
	\end{equation}
	 belongs to the normal space at $x$ and depends smoothly on $x$. Hence, the function $(I - \Pi)\D _W Z$ is a smooth normal vector field on $\mcal$. 
	\begin{definition}
		The second fundamental form  $\II:\mathfrak{X}_{\mathrm{T}}(\mcal)\times\mathfrak{X}_{\mathrm{T}}(\mcal)\to\mathfrak{X}_{\mathrm{N}}(\mcal)$ is a symmetric bilinear map defined by $\II(W,Z) = (I - \Pi)\D _W Z$.
	\end{definition}  
	\begin{definition}
		The Weingarten map is 
		\begin{equation}
			\map{\wcal}{\mathfrak{X}_{\mathrm{T}}(\mcal)\times\mathfrak{X}_{\mathrm{N}}(\mcal)}{\mathfrak{X}_{\mathrm{T}}(\mcal)}{(W,N)}{\wcal(W,N)}
		\end{equation}
		defined by the multilinear form $\scalp{N}{\II(W,Z)} = \scalp{\wcal(W,N)}{Z}, \quad \forall\, Z\in \mathfrak{X}_{\mathrm{T}}\pa{\mcal}$.
	\end{definition}
	Since $\nabla_WZ$ can be shown to depend only on the value of the vector field $W$ at $x$,  using the properties of the Levi-Civita connection both the second fundamental form and the Weingarten map can be defined pointwise, i.e. depending only on values of the vector fields at $x$~\cite[Proposition 8.1]{leeRiemManifs2}. Given $w,z\in T_x\mcal$ and $n\in N_x\mcal$, the Weingarten map $\wcal_x:T_x\mcal\times N_x\mcal \to T_x\mcal$ is defined pointwise as
	\begin{equation}
		 \scalp{\wcal_x(w,n)}{z}_x = \scalp{\wcal(W,N)}{Z}\rest{x}
	\end{equation}
    for any $W,Z\in \mathfrak{X}_{\mathrm{T}}(\mcal)$ and $N\in\mathfrak{X}_{\mathrm{N}}(\mcal)$ satisfying $W(x) = w$, $Z(x) = z $ and $N(x) = n$. The following characterization relates the Weingarten map at $x$ with the differential of the tangent space projection.  
	\begin{proposition}[{\cite[Theorem 1]{extrinsiclook}}]\label{prop:weingartenAndDifferentialOfProjection}
		For any $x\in\mcal$, $w\in T_x\mcal$ and $n\in N_x\mcal$, the Weingarten map satisfies 
		\begin{equation}
			\wcal_x(w,n) = \Diff{\Pi}{x}{w}n = \Pi(x)\Diff{\Pi}{x}{w} n = \Pi(x)\Diff{\Pi}{x}{w}\Pi^{\perp}(x) q
		\end{equation}
		for any $\revision{q}\in T_x\ecal\simeq\ecal$ such that $\Pi^{\perp}(x) \revision{q} = n$.
	\end{proposition}
	
\paragraph{Expression for the fixed-rank manifold.}
 Depending on the conventions used to represent points and tangent vectors, many equivalent expressions are known for the Weingarten map of the fixed-rank manifold, see for example~\cite{extrinsiclook,fepponLermusiaux18}. In our conventions, for any $Y = U\Sigma V^\top\in\mcal_r$, ${T = UMV^\top + U_{\mathrm{p}}V^\top + UV_{\mathrm{p}}^\top\in T_Y\mcal_r}$ and $N\in N_Y\mcal_r$ the Weingarten map can be computed as 
	\begin{equation}
		\wcal_Y(T,N) = NV_{\mathrm{p}} \Sigma^{-1} V^\top  +U \revision{\Sigma^{-1}}U_{\mathrm{p}}^\top N\in T_Y\mcal_r.
	\end{equation}
	This expression nicely highlights two notable features that are known for the fixed-rank manifold: it is a ruled surface with unbounded curvature.  In fact, along the subspace associated to the $UMV^\top$ term, $\mcal_r$ is flat, while the curvature along the other directions grows unbounded as $\sigma_r(Y) \rightarrow 0$.
		
	\subsubsection{The accelerated forward Euler (AFE) integration scheme for DLRA}\label{sss:afeScheme}
	With the definitions introduced in Section~\ref{sss:weingartenMap}, we can compute the acceleration of the DLRA solution curve allowing to generalize the AFE integration scheme to Problem~\ref{problem:projectedDynamicsProblem}. 
	\begin{proposition}\label{prop:dlraAcceleration}
		If a smooth curve on $\mcal_r$ is defined by $\dot Y = \Pi(Y) F(Y)$ for some smooth vector field $F:\Rmn\to\Rmn$, then its intrinsic acceleration can be computed as
		\begin{equation}\label{eq:dlraAcceleration}
			\ddot{Y} = \Pi(Y)\Diff{F}{Y}{\Pi(Y)F(Y)} + \wcal_Y(\Pi(Y) F(Y), \Pi^\perp(Y) F(Y)).
		\end{equation}
	\end{proposition}
	\begin{proof}
		By definition of the Levi-Civita connection for Riemannian submanifolds we have
		\begin{equation}
			\ddot Y = \nabla_{\dot{Y}}\dot{Y} = \Pi(Y) \D_{\dot Y}\pa{\Pi(Y)F(Y)}.
		\end{equation}	
		Using the product rule and the fact that $\Pi(Y)^2 = \Pi(Y)$, it follows from  the last equality of Proposition~\ref{prop:weingartenAndDifferentialOfProjection} that
		\begin{align*}
			\Pi(Y) \D_{\dot Y}\pa{\Pi(Y)F(Y)} &= \Pi(Y)\pa{\Pi(Y)\D F(Y)[\dot Y] + \Pi(Y)\D \Pi(Y)[\dot Y] F(Y)}\\
			&= \Pi(Y)\Diff{F}{Y}{\Pi(Y)F(Y)} + \Pi(Y)\Diff{\Pi}{Y}{\Pi(Y)F(Y)}\Pi^\perp(Y) F(Y)\\
			&= \Pi(Y)\Diff{F}{Y}{\Pi(Y)F(Y)} + \wcal_Y(\Pi(Y) F(Y), \Pi^\perp(Y) F(Y)).
			\tag*{$\qed$}
		\end{align*}
\end{proof}
	
	Hence, to mimic the Euclidean AFE update defined by equation~\eqref{eq:euclideanAFE}, we need to construct a smooth manifold curve $Y^\mathrm{AFE}(\Deltat)$ such that
	\begin{equation*}
		Y^{\mathrm{AFE}}(0) = Y_0,\quad \dot Y^{\mathrm{AFE}}(0) = \dot Y(0),\quad \ddot Y^{\mathrm{AFE}}(0) = \ddot Y(0).
	\end{equation*}
	Proposition~\ref{prop:retractionCurveWithAcceleration} shows that we can construct such a curve using a second-order retraction. Let $\Retr^{\II}$ indicate any second-order retraction, then the manifold analogous of the AFE update~\eqref{eq:euclideanAFE} reads as
	\begin{equation}\label{eq:generalAFEUpdate}
		Y^{\mathrm{AFE}}\pa{\Deltat} = \Retr_{Y_0}^{\II}\pa{\Deltat \dot Y_0 +  \Deltat^2 \ddot Y_0 / 2}.
	\end{equation}
	Then, the AFE scheme for DLRA takes the form 
	\begin{equation}\label{eq:dlraAFEscheme}
		\revision{Y_{k+1}} = \Retr_{Y_k}^{\II}\bigg(\Deltat \Pi(Y_k) F(Y_k) +  \frac{\Deltat^2}{2}\ddot Y_{k}  \bigg)
	\end{equation}
	with 
	\begin{equation}\label{eq:accelerationDLRAatStepK}
		\ddot{Y}_k = \Pi(Y_k)\Diff{F}{Y_k}{\Pi(Y_k)F(Y_k)} + \wcal_{Y_k}\pa{ \Pi(Y_k) F(Y_k),  \Pi^\perp(Y_k) F(Y_k)}.
	\end{equation}	
	As a consequence of~\eqref{eq:approximationPowerOfRetractionCurves}, this scheme admits a local truncation error of order $O(\Deltat^3)$.

	All the retractions for the fixed-rank manifold presented in Section~\ref{s:retractionsAndDLRA}	 have the second-order property. In principle, all of them are suited to be used as $\Retr^{\II}$ in~\eqref{eq:dlraAFEscheme}, however, experiments reported in Section~\ref{s:numericalExperiments} suggest the orthographic retraction is the most convenient in terms of speed, accuracy and stability. 		
	
	\subsection{The \revision{Projected Ralston--Hermite} integration scheme}\label{ss:rh}
	
	As put forth in~\cite{mySecondPaper}, retractions can also be used to define a manifold curve with prescribed endpoints and endpoint velocities. For tangent bundle data points $(x_0,v_0),(x_1,v_1)\in T\mcal$ and some parameters $t_0<t_1$, we denote by $H$ the retraction-based Hermite (RH) interpolant defined by~\cite[Corollary 7]{mySecondPaper}, which satisfies
	\begin{equation}\label{eq:HermiteInterpolant}
		H(t_i) = x_i, \quad \dot H(t_i) = v_i,\quad i = 0,1.
	\end{equation} 
	We will employ the notation $H(t;(p_0,v_0,t_0),(p_1,v_1,t_1))$ to underline the dependence of $H$ on the interpolation data.
	The curve $H$ is constructed by a manifold extension of the well-known De Casteljau algorithm. But instead of using geodesic segments as building blocks of the procedure, the RH interpolant is defined with endpoint curves constructed with the use of a retraction and its local inverse, making the method more efficient and more broadly applicable. An efficient evaluation of the curve $H$ requires a retraction for which the inverse retraction admits a computationally affordable procedure to evaluate. 
	On the fixed-rank manifold, the orthographic retraction is a suitable candidate to construct the RH interpolant. 	
	
	As stated in~\cite[Theorem 17]{mySecondPaper}, the RH interpolant achieves $O(\Deltat^4)$ approximation error as $\Deltat\to0$ in the case where the interpolation data is sampled from \revision{a} smooth manifold curve at $t_0$ and $t_1 = t_0 + \Deltat$.  The \revision{Projected Ralston--Hermite} (\revision{PRH}) scheme aims at leveraging this approximation power to define a numerical integration update rule with small local truncation error. Let us derive the \revision{PRH} method for the initial value problem~\eqref{eq:ambientEquation} evolving in $\Rmn$. 

	\subsubsection{Euclidean analog}
	Given $A_0,A_1,V_0,V_1\in\Rmn$ and $t_0<t_1$, the Hermite interpolant $\tau \mapsto H_{\mathrm{P}}(\tau; (t_0,A_0,V_0), (t_1,A_1,V_1))$ is the unique third degree polynomial curve in $\Rmn$  which satisfies $H_\mathrm{P}(t_i) = A_i$, $H_\mathrm{P}'(t_i) = V_i$, for $i = 0,1$. The polynomial  curve $H_\mathrm{P}$ can be used to define the following multistep integration scheme for the initial value problem~\eqref{eq:ambientEquation}:
	\begin{equation}\label{eq:notoriouslyUnstableScheme}
		A_{k+2} = H_{\mathrm{P}}\pa{t_{k+2}; \pa{t_k,A_k, F(A_k)}, \pa{t_{k+1},A_{k+1}, F(A_{k+1})}}.
	\end{equation}
	Working out the expression for $H_{\mathrm{P}}$ allows rewriting the recursive relation~\eqref{eq:notoriouslyUnstableScheme} as
	\begin{align}
		A_{k+2} &= 5 A_k - 4 A_{k+1} + \Deltat\pa{2F(A_{k}) + 4 F(A_{k+1})}\\
		&= A_{k+1} + \Deltat \pa{2F(A_k) + 4 F(A_{k+1}) - 5\pa{\frac{A_{k+1} - A_{k}}{\Deltat}}}.
	\end{align}
	As pointed out in~\cite[\S III.3]{hairerSolveODEs1}, this scheme has a local truncation error $O(\Deltat^4)$, the highest possible order for explicit 2-step method using these terms. However, it is not zero-stable and thus does not produce a convergent scheme of order 3. Nevertheless, this update rule can be combined with suitably chosen intermediate steps to recover stability.
	Consider the family of multistep methods obtained by taking an intermediate forward Euler step of length $\alpha\in\pa{0,1}$ combined with the update rule~\eqref{eq:notoriouslyUnstableScheme} as 
	\begin{equation}\label{eq:ralstonAsEulerAndHermite}
		\begin{cases*}
			A_{k+\alpha} = A_{k} + \alpha \Deltat F(A_k),\\
		A_{k+1} = H_{\mathrm{P}}(t_{k+1}; \pa{t_k,A_k, F(A_k)}, \pa{t_{k+\alpha},A_{k+\alpha}, F(A_{k+\alpha})}).
		\end{cases*}
	\end{equation}
     These schemes are all part of the family of explicit Runge--Kutta methods as it can be shown that
	\begin{equation}
		A_{k+1} = A_k + \Deltat\pac{\pa{1 + \frac{1}{\alpha} - \frac{1}{\alpha^2}} F(A_k) + \pa{\frac{1}{\alpha^2}-\frac{1}{\alpha}}F(A_{k+\alpha})}.
	\end{equation}
	For any $\alpha\in\pa{0,1}$, the scheme satisfies the first-order conditions of Runge--Kutta methods. Choosing $\alpha$ to satisfy also the second-order conditions narrows down the family to the scheme with $\alpha = 2/3$. This scheme is an explicit second-order Runge--Kutta method known as the Ralston scheme and is defined by the following Butcher table. 
	\[
	\renewcommand\arraystretch{1.2}
	\begin{array}
		{c|cc}
		0\\
		\frac{2}{3} & \frac{2}{3}\\
		\hline
		& \frac{1}{4} &\frac{3}{4} 
	\end{array}
	\] 

	\subsubsection{The \revision{Projected Ralston--Hermite} (\revision{PRH}) integration scheme for DLRA}
	Expressed in the form~\eqref{eq:ralstonAsEulerAndHermite}, the update rule of the Ralston scheme consists of a forward Euler step followed by a Hermite interpolation step. We can leverage this formulation and the RH interpolant~\eqref{eq:HermiteInterpolant} to extend the Ralston method to the manifold setting. Let $\Retr$ denote any retraction and $\Retr^{\mathrm{I}}$ denote any retraction that can be used in practice to evaluate the RH interpolant (i.e. whose inverse can be efficiently computed). To indicate which retraction is used to construct the retraction-based Hermite interpolant $H$, we add it to its list of arguments. The \revision{Projected Ralston--Hermite} (\revision{PRH}) scheme for Problem~\ref{problem:projectedDynamicsProblem} is then defined as 
	\begin{equation}\label{eq:rhScheme}
		 \begin{cases*}
		 	Y_{k+2/3} = \Retr_{Y_{k}}\pa{\frac{2}{3}\Pi(Y_k) F(Y_k) },\\
		 	\revision{Y_{k+1}} = H (t_{k+1}; \pa{t_k,Y_k,\Pi(Y_k)F(Y_k)}, \pa{t_k + \frac{2}{3}\Deltat,Y_{k+2/3},\Pi(Y_{k+2/3})F(Y_{k+2/3})}, \Retr^{\mathrm{I}} ).
		 \end{cases*}
	\end{equation}
	A suitable candidate for both retractions is $\Retr = \Retr^{\mathrm{I}} = \Retr^{\mathrm{ORTH}}$, the orthographic retraction. As experiments in Section~\ref{s:numericalExperiments} suggest, this generalization to the manifold setting of the Ralston scheme maintains its second-order accuracy.
   	\subsubsection{The A\revision{PRH} integration scheme}\label{ss:arh}
  	For the sake of completeness, a third scheme can be obtained by combining the AFE and the \revision{PRH} schemes. Replacing the intermediate forward Euler step of the \revision{PRH} scheme with an AFE update defines what we call the Accelerated \revision{Projected Ralston--Hermite} scheme (A\revision{PRH}). It is defined by the recursive relation 
  	\begin{equation}\label{eq:arhScheme}
  		\begin{cases*}
  			Y_{k+2/3} = \Retr_{Y_k}^{\II}\bigg(\frac{2}{3}\Deltat \Pi(Y_k) F(Y_k) +  \frac{2\Deltat^2}{9}\ddot Y_{k}  \bigg),\\
  			\revision{Y_{k+1}} = H (t_{k+1}; \pa{t_k,Y_k,\Pi(Y_k)F(Y_k)}, \pa{t_k + \frac{2}{3}\Deltat,Y_{k+2/3},\Pi(Y_{k+2/3})F(Y_{k+2/3})}, \Retr^{\mathrm{I}} ),
  		\end{cases*}
  	\end{equation}
  	where $\ddot Y_k$ is given by~\eqref{eq:accelerationDLRAatStepK}.

	\section{Numerical experiments}\label{s:numericalExperiments}
	In this section, we illustrate the performances of the accelerated forward Euler (AFE) method, the \revision{Projected Ralston--Hermite} (\revision{PRH}) method and the accelerated \revision{Projected Ralston--Hermite} (A\revision{PRH}) method. Experiments were executed with Matlab 2022b on a laptop computer with Intel i7 CPU (1.8GHz with single-thread mode) with 8GB of RAM, 1MB of L2 cache and 8MB of L3 cache. The implementation leverages the differential geometry tools of the Manopt library~\cite{manoptPaper}. In particular, the orthographic retraction provided by  Manopt is used for the implementation of AFE, \revision{PRH} and A\revision{PRH}. An implementation of the KSL and KLS retractions as described by Algorithms~\ref{alg:retractionKSL} and~\ref{alg:retractionKLS} were added to the fixed-rank manifold factory. For the implementation of the projected Runge--Kutta method of~\cite{kieriVandereycken}, we also added an implementation of the truncated SVD extended retraction, accepting as input tangent vectors of arbitrary tangent spaces. \revision{See Table~\ref{tab:recapOfMethods} for a summary of the acronyms of the methods appearing in the numerical experiments. For the sake of completeness, we recall that an $s$-stage Projected Runge--Kutta (PRKs) method for given Butcher table, as provided in \cite{kieriVandereycken}, is defined as follows  
    \begin{equation*}
        \text{(PRKs) }
        \begin{cases}
        Y_{k+1} = \mathcal R(Y_k +h \sum_{j=1}^s b_j \Pi(\mathcal R(Z_j)) F(\mathcal R(Z_j)))\, , \\
        Z_j = Y_k + h\sum_{l=1}^{j-1} \Pi( \mathcal R(Z_l)) F(\mathcal R(Z_l)) \, ,
        \end{cases}
    \end{equation*}
    where $\mathcal R$ denotes the metric projection on the fixed-rank manifold $\mathcal M_r$. We refer to \cite[\S 5]{kieriVandereycken} for a detailed description of its efficient implementation, along with a recapitulation of the Butcher table for each $s$-stage method up to the third order.}
	
	\begin{table}[]
		\centering 
		\caption{Summary of the acronyms of the methods considered in the numerical experiments. }
		\begin{tabular}{|c|c|}
			\hline
			KSL & Projector-splitting integrator~\cite{absilOseledets} \\ \hline
			KLS & Unconventional integrator~\cite{unconventionalIntegrator} \\ \hline
			PRKs ($s = 1,2,3$) & Projected Runge--Kutta methods~\cite{kieriVandereycken} \\ \hline
			AFE & Accelerated forward Euler method, Section~\ref{ss:afe} \\ \hline
			\revision{PRH} & \revision{Projected Ralston--Hermite} method, Section~\ref{ss:rh} \\ \hline
			A\revision{PRH} & Accelerated \revision{Projected Ralston--Hermite} method, Section~\ref{ss:arh}\\ \hline
		\end{tabular}
		\label{tab:recapOfMethods}
	\end{table}
	
	\subsection{Differential Lyapunov equation}\label{ss:lyapunovExperiments}
	The modeling error~\eqref{eq:modelingError} introduced by DLRA is associated with the normal component of the vector field of the original differential equation. The effect of this error on the performance of DLRA integrators can be assessed by considering a class of matrix differential equations, the so-called differential Lyapunov equations~\cite[\S 6.1]{bartUschmajewChapter}, which take the form
	\begin{equation}\label{eq:diffLyapunovEquation}
		\begin{cases*}
			A'(t) = L A(t) + A(t) L^\top  + Q,\quad t\in\pac{0,T},\\
			A(0) = A_0,
		\end{cases*}
	\end{equation}
	for some $A_0, L, Q\in \Rnn$. Indeed, if $A_0$ has rank exactly $r$ and the matrix $Q$ is zero, then $A(t)$ is also of rank-$r$ for every $t\in\pac{0,T}$~\cite[Lemma 1.22]{helmkeMoore94}. Therefore, the norm of $Q$ is proportional to the modeling error. 
	
	\begin{figure}
		\centering
		\begin{minipage}{0.475\linewidth}
			\centering
			\centering \includegraphics[trim=0cm 0cm 0cm 0.75cm, clip, width=0.975\linewidth]{ 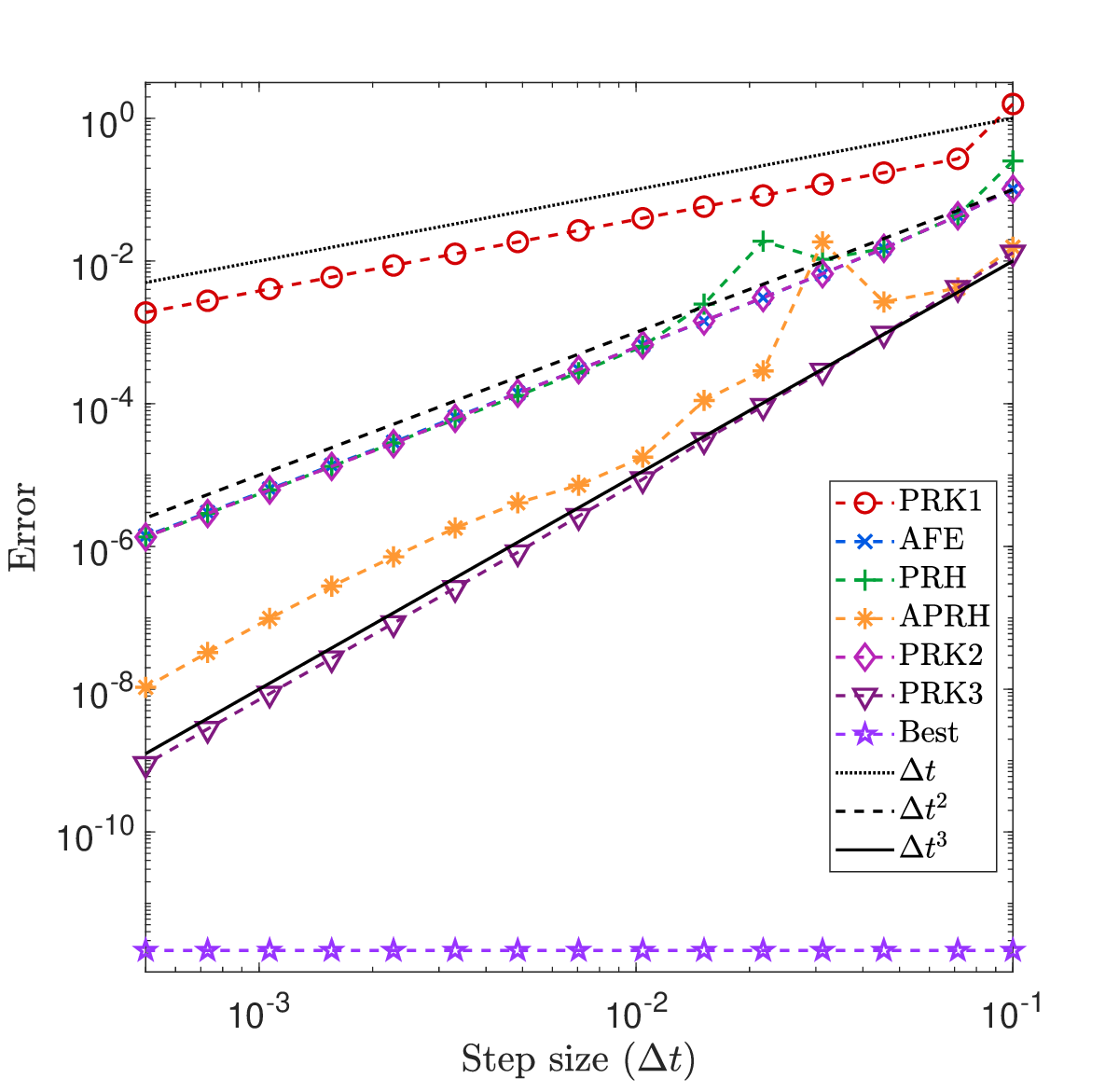}\vspace*{-0.0cm}
			\centering \includegraphics[trim=0cm 0cm 0cm 0.0cm, clip, width=0.975\linewidth]{ 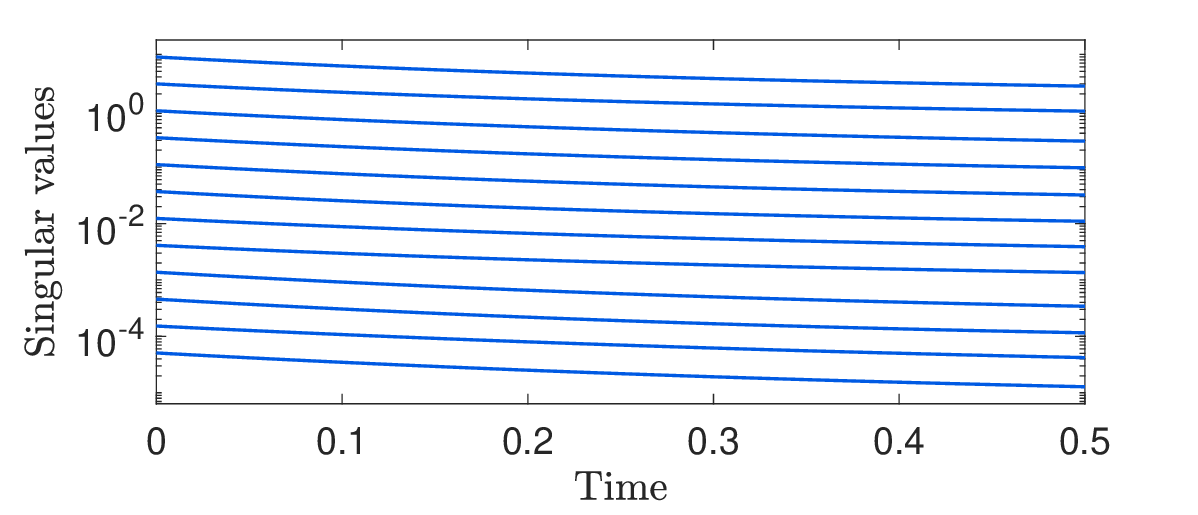}	
			\subcaption{$\norm{Q}_{\mathrm{F}} = 0$.}
		\end{minipage}
		\begin{minipage}{0.475\linewidth}
			\centering
			\centering \includegraphics[trim=0cm 0cm 0cm 0.75cm, clip, width=0.975\linewidth]{ 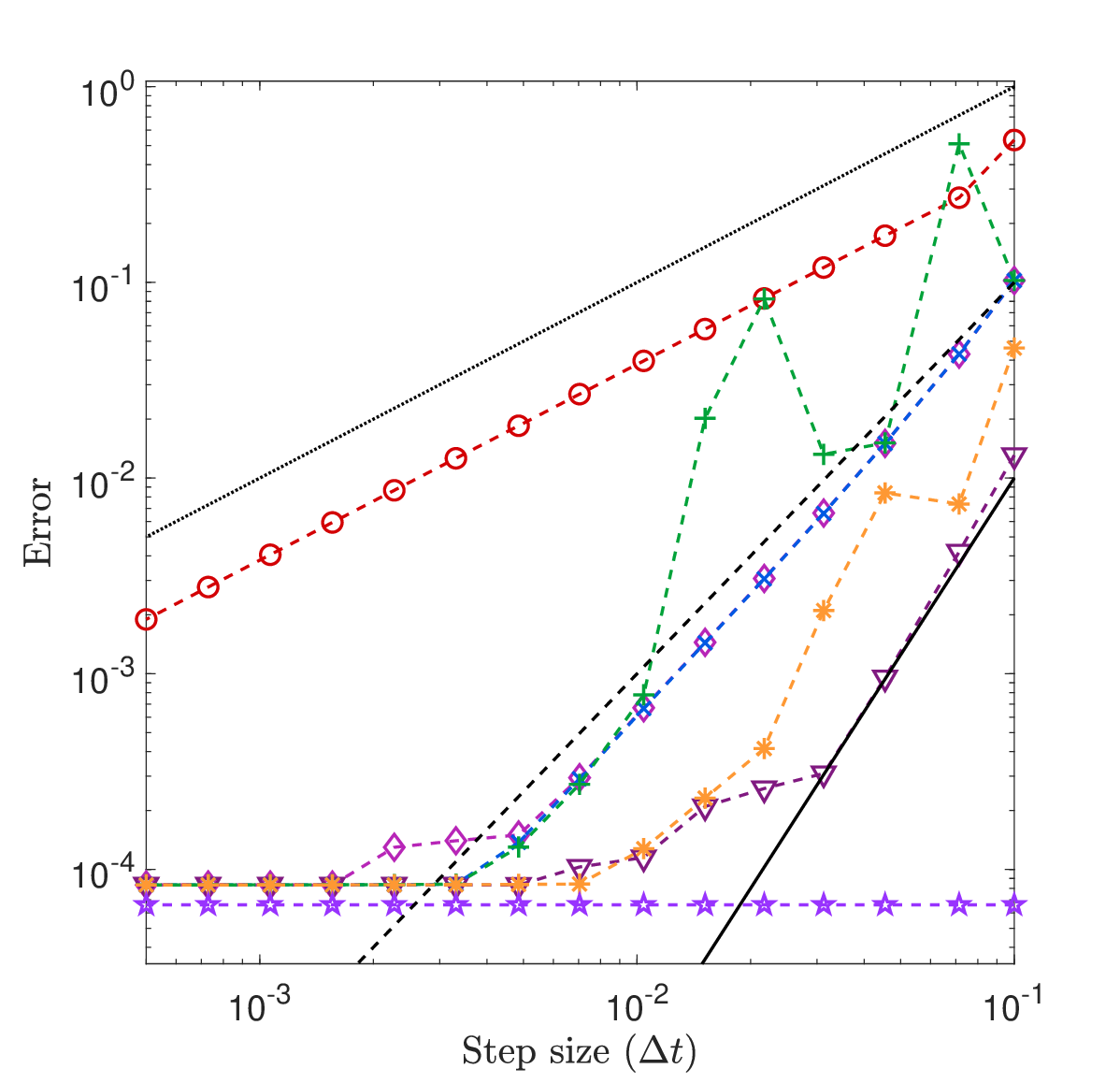}\vspace*{-0.0cm}
			\centering \includegraphics[trim=0cm 0cm 0cm 0.0cm, clip, width=0.975\linewidth]{ 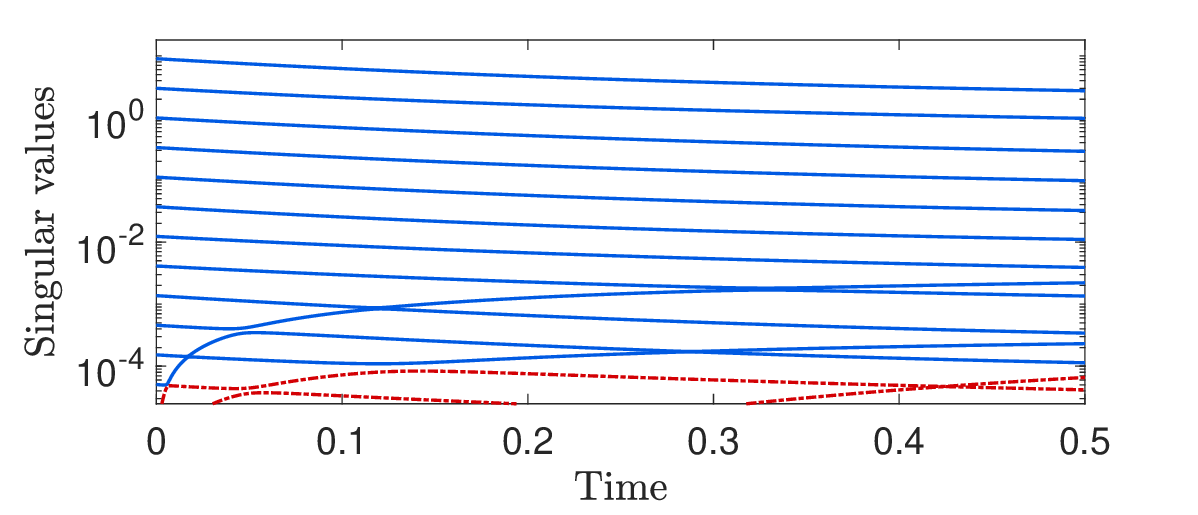}	
			\subcaption{$\norm{Q}_{\mathrm{F}} = 0.01$.}
		\end{minipage}
		
		\begin{minipage}{0.475\linewidth}
			\centering
			\centering \includegraphics[trim=0cm 0cm 0cm 0.75cm, clip, width=0.975\linewidth]{ 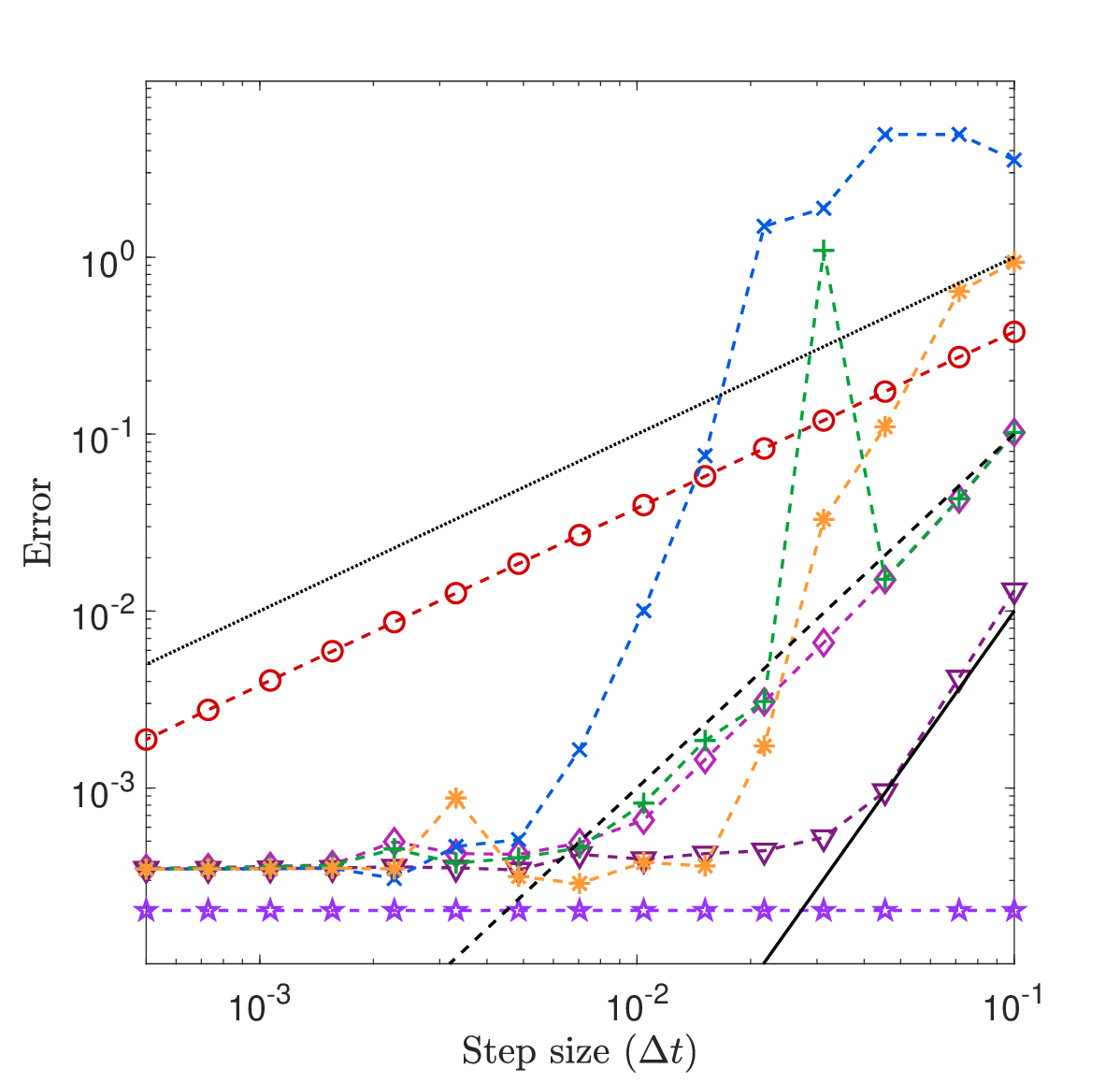}\vspace*{-0.0cm}
			\centering \includegraphics[trim=0cm 0cm 0cm 0.0cm, clip, width=0.975\linewidth]{ 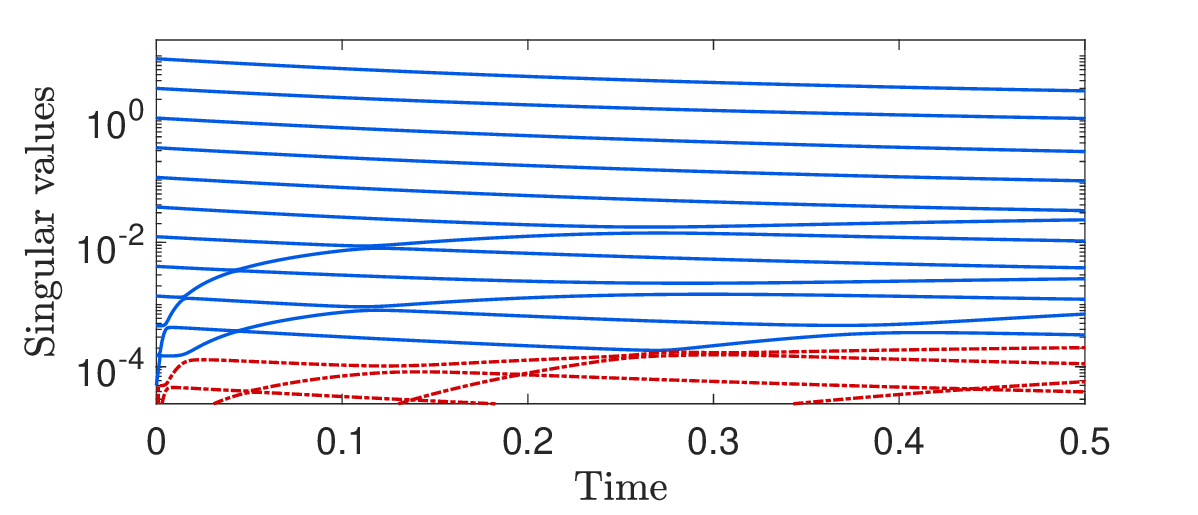}	
			\subcaption{$\norm{Q}_{\mathrm{F}} = 0.1$.}
		\end{minipage}
		\begin{minipage}{0.475\linewidth}
			\centering
			\centering \includegraphics[trim=0cm 0cm 0cm 0.75cm, clip, width=0.975\linewidth]{ 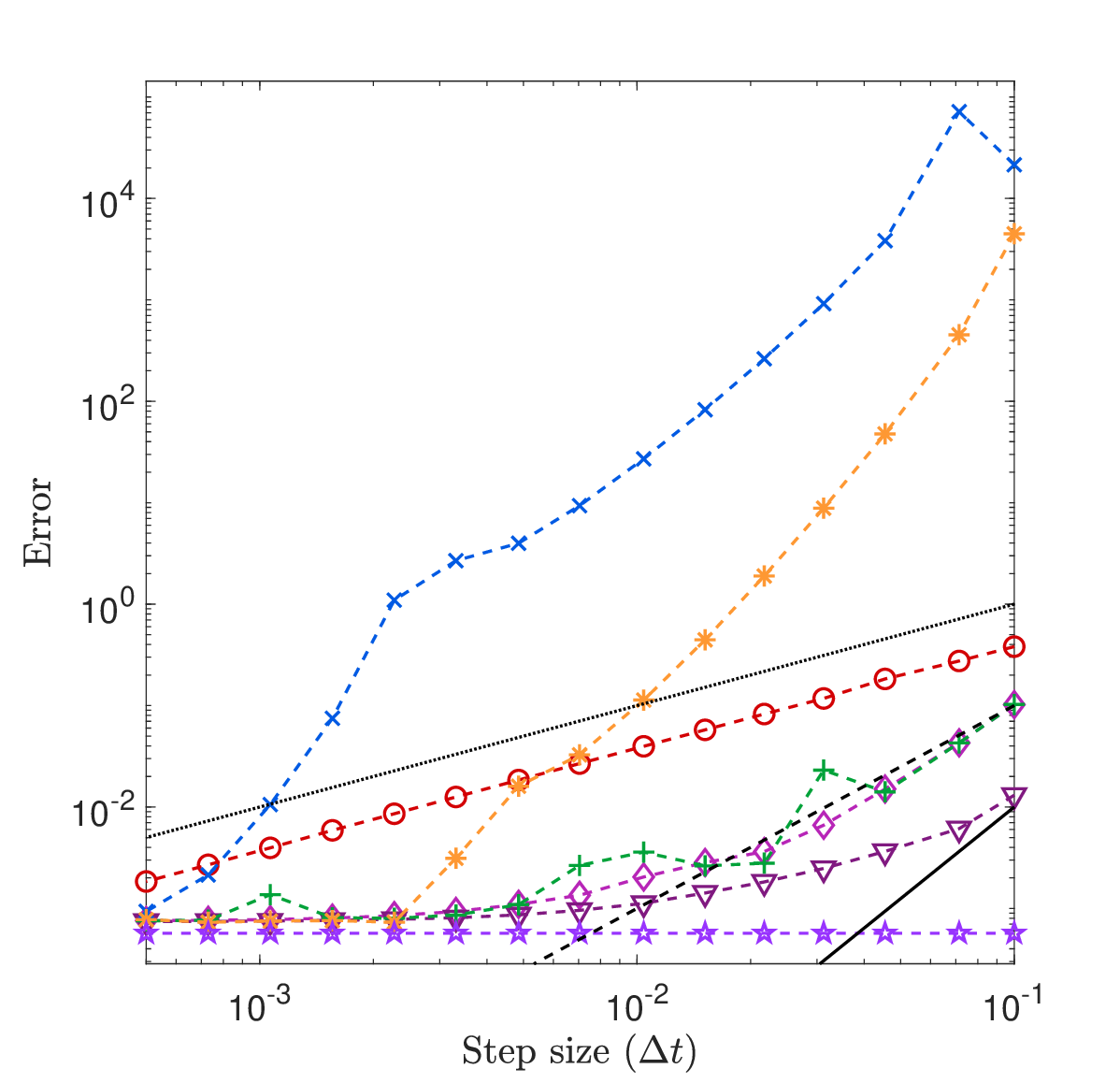}\vspace*{-0.0cm}
			\centering \includegraphics[trim=0cm 0cm 0cm 0.0cm, clip, width=0.975\linewidth]{ 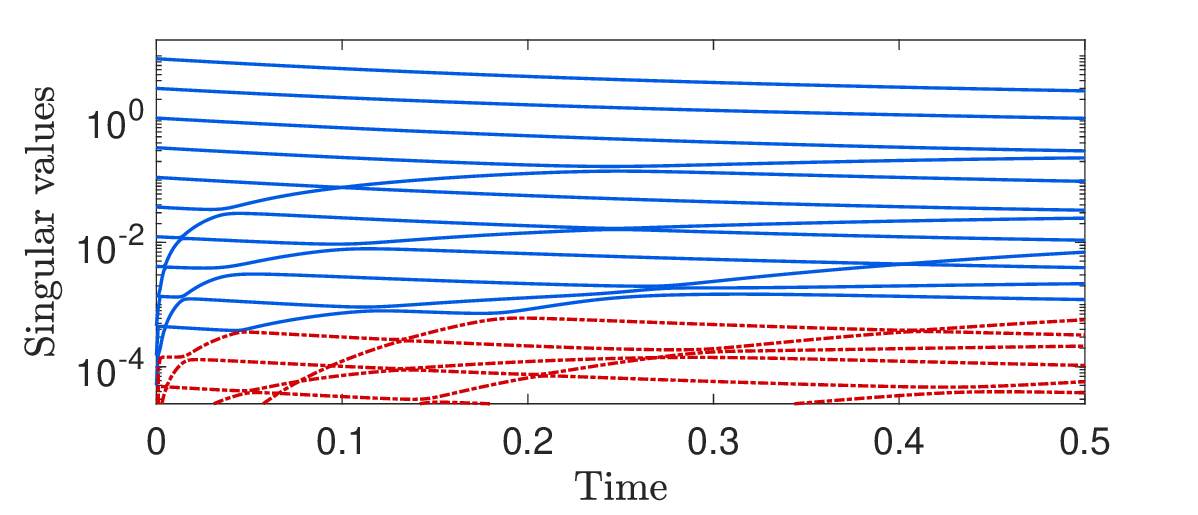}	
			\subcaption{$\norm{Q}_{\mathrm{F}} = 1$.}
		\end{minipage}
		\caption{Convergence of the error at final time for different DLRA integration schemes applied to the Lyapunov equation~\eqref{eq:diffLyapunovEquation} with sources terms of different norms. The top plot in each panel is the final error $\norm{Y_{\Deltat}(T)-A(T)}_2$ versus the step size ${\Deltat}$, where $Y_{\Deltat}$ is the approximation of $A$ obtained with a step size ${\Deltat}$. The bottom plot reports the evolution of the singular values of the reference solution over time. The red dashed curves correspond to discarded singular values.}
		\label{fig:lyapunovResultsConvergence}
	\end{figure}

	\begin{table}
		\centering
		\caption{Average time in milliseconds per step for the experiments of Figure~\ref{fig:lyapunovResultsConvergence}-(d).}
		\begin{tabular}{|c|c|c|c|c|c|c|c|}
			\hline
			PRK1 & KSL & KLS & AFE & \revision{PRH} & A\revision{PRH} & PRK2 & PRK3 \\ \hline
			$5.27$ & $5.41$ & $5.88$ & $7.71$ & $12.62$ & $17.38$ & $10.41$ & $14.16$ \\ \hline
		\end{tabular}
		\label{tab:lyapunovResultsTimePerIteration}
	\end{table}
	
	In the following, $L$ is the discretized 1D Laplace operator on a uniform grid, that is, $L$ is the tridiagonal matrix with $-2$ on the main diagonal and $1$ on the first off-diagonals. For the source term, we take $Q = \eta \tilde Q / \normSmall{\tilde Q}_{\mathrm{F}}$ for some $\eta>0$ and where $\tilde Q$ is a full rank matrix generated from its singular value decomposition with randomly chosen singular vectors and prescribed singular values, decaying as $\sigma_i(\tilde Q) = 10^{2-i}$, for $i = 1,\dots, n$. The initial condition is taken to be of rank exactly $r$ and is also assembled from a randomly generated singular value decomposition with a prescribed geometric decay of non-zero singular values: $\sigma_i(A_0) = 3^{2-i}$, for all $i = 1,\dots,r$, and $\sigma_i(A_0) = 0$, for all $i = r+1,\dots,n$. 
	
	In Figure~\ref{fig:lyapunovResultsConvergence}, we report the results with $n = 100$ and $r = 12$ of the following experiments. For different values of $\eta$, we numerically integrate the rank-$r$ DLRA differential equation~\eqref{eq:odeDLRA} applied to~\eqref{eq:diffLyapunovEquation} with different numerical schemes and different time steps up to $T = 0.5$. A reference solution to the ambient equation~\eqref{eq:diffLyapunovEquation} is found by using the MATLAB routine \texttt{ode45} between each time step, for a time step that is the smallest among those considered for the numerical integrators. We then plot as a function of the step size the 2-norm discrepancy between the reference solution at final time and its approximation obtained by numerical integration. The numerical results for the KSL and the KLS scheme were very similar to the ones of PRK1. Hence, they were omitted not to overcrowd the plots.

	The panels of Figure~\ref{fig:lyapunovResultsConvergence} correspond to the cases (a) $\eta = 0$, (b) $\eta = 0.01$, (c) $\eta = 0.1$, (d) $\eta = 1.0$. When the source term is zero, the reference solution is also of rank exactly $r$, as can be seen from the value of the best approximation error in panel (a). In this regime, the AFE and the \revision{PRH} scheme both exhibit $O({\Deltat}^2)$ error convergence, while the A\revision{PRH} scheme seem to reach an asymptotic $O({\Deltat}^3)$ trend. The trade-off between accuracy and computational effort that can be seen in Figure~\ref{fig:lyapunovResultsEffort}-(a) shows that in this simple setting, the \revision{PRH}, AFE and A\revision{PRH} schemes have comparable performances to PRK2.
	Turning on the source term determines a non-negligible best approximation error due to the growth of singular values that were initially zero, as can be seen in the bottom plots of panels (b), (c) and (d). The larger the source term's norm, the faster and the greater these singular values grow. Then, the numerical integrators converge to the exact solution of the projected system and so the error with respect the ambient solutions stagnates at a value slightly higher than the best 2-norm approximation. While the \revision{PRH} scheme preserve the $O({\Deltat}^2)$ trend up to some oscillations as $\eta$ increases, the AFE and A\revision{PRH} schemes seem to suffer instability when the normal component of the vector field is too large. 
	In this more realistic scenario where the normal component of the vector field is non-negligible, only the \revision{PRH} scheme remains comparable to PRK2 in terms of the trade-off between accuracy and effort, see Figure~\ref{fig:lyapunovResultsEffort}-(b) and Table~\ref{tab:lyapunovResultsTimePerIteration}.

	\begin{figure}
		\centering
		\begin{minipage}{0.475\linewidth}
			\centering \includegraphics[trim=0cm 0cm 0cm 0cm,clip, width = 0.9\linewidth]{ 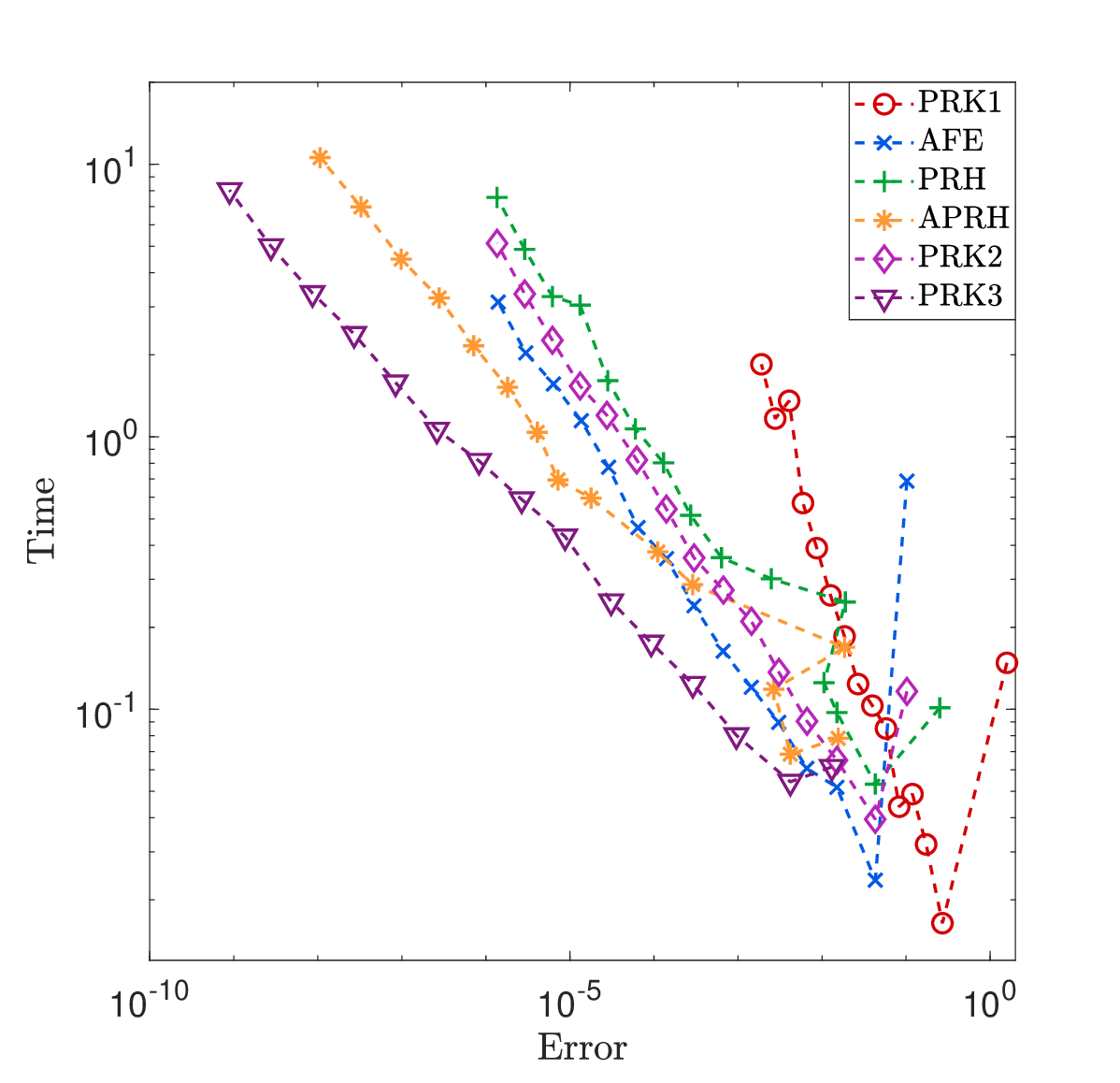}
			\subcaption{$\norm{Q}_{\mathrm{F}} = 0$.}
		\end{minipage}
		\begin{minipage}{0.475\linewidth}
			\centering \includegraphics[trim=0cm 0cm 0cm 0cm clip, width = 0.9\linewidth]{ 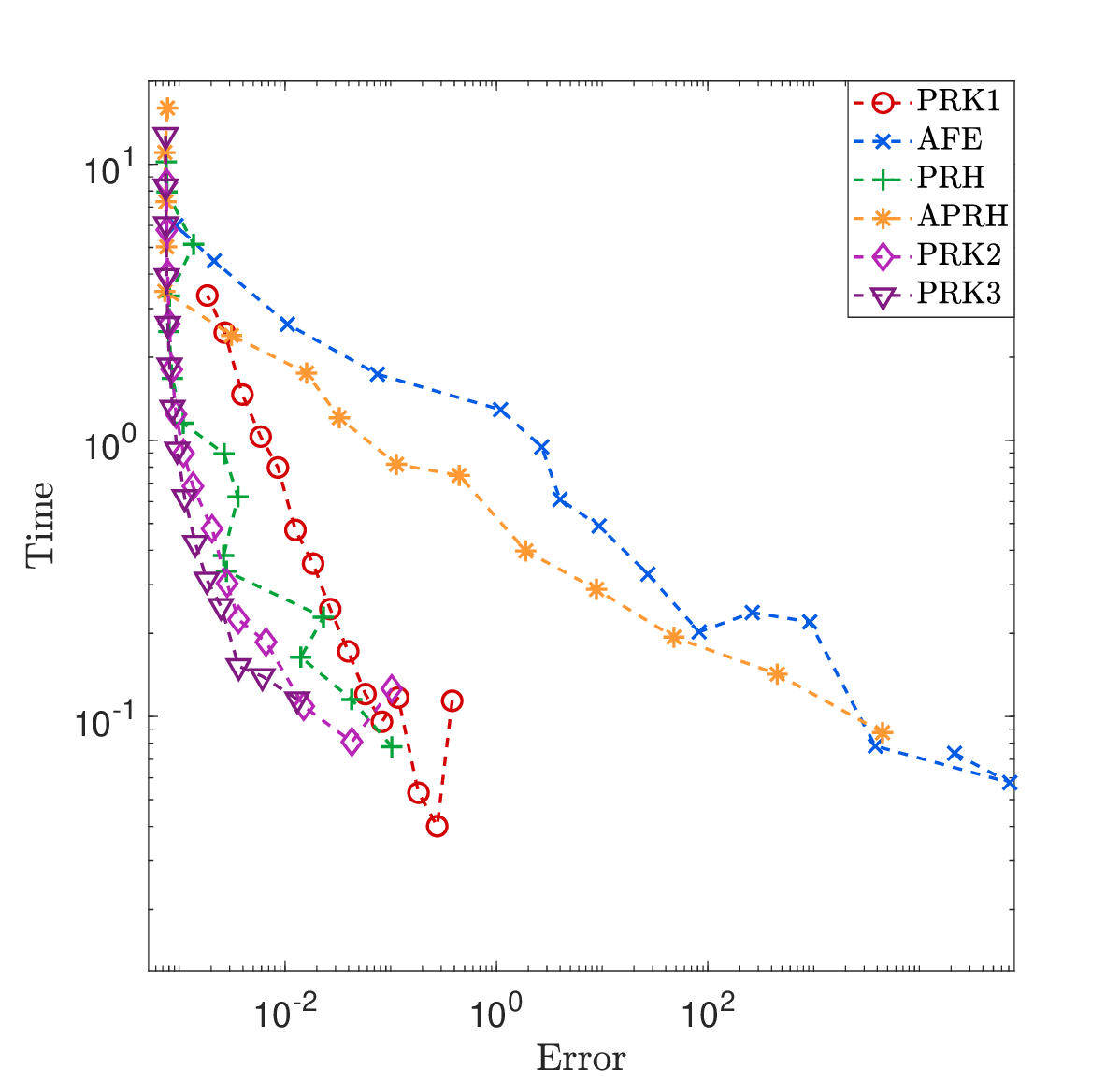}
			\subcaption{$\norm{Q}_{\mathrm{F}} = 1$.}
		\end{minipage}
		\caption{Computational effort in terms of wall-clock time against the error with respect to the reference solution achieve by different numerical integration schemes and with different step sizes. The results were collected from the same experiment of panels (a) and (d) of Figure~\ref{fig:lyapunovResultsConvergence}.}
		\label{fig:lyapunovResultsEffort}
	\end{figure}

	\subsection{Robustness to small singular values}\label{ss:robustness-to-small-singular-values}
	A fundamental prerequisite for competitive DLRA integrators is to be resilient to the presence of small singular values in the solution. A detailed discussion on the topic can be found in~\cite{kieriLubichWallach}. In applications, very often the ambient solution admits an exponential decay of singular values. Hence, a good low-rank approximation is possible but the occurrence of small singular values is inevitable for DLRA to be accurate: a rank-$r$ approximation of the solution must match the $r$th singular value of the ambient solution, which is small if the approximation error is small.

	The smaller the singular values of the solution, the greater the stiffness of the DLRA differential equation~\eqref{eq:odeDLRA}: the Lipschitz constant of the vector field $F$ gets multiplied by the Lipschitz constant of the tangent space projection, which is inversely proportional to the smallest non-zero singular value of the base point~\cite[Lemma 4.2]{kochLubich}. Accordingly, standard numerical integration methods fail to provide a good approximation unless the step size is taken to be very small. Projector-splitting integrators for DLRA avoid such step size restrictions as the error convergence results are independent of the smallest non-zero singular value of the approximation. These \revision{schemes} are commonly qualified as robust to small singular values. The robustness property was shown for the KSL scheme~\cite[Theorem 2.1]{kieriLubichWallach} and the KLS scheme~\cite[Theorem 4]{unconventionalIntegrator}. The PRK methods also enjoy the robustness property~\cite[Theorem 6]{kieriVandereycken}. In the following, we experimentally study the robustness of the AFE, \revision{PRH} and A\revision{PRH} integration schemes to the presence of small singular values. 
	
	The typical setting to assess the stability to small singular values of a given integration scheme, considers a matrix curve $t\in\pac{0,T}\revision{\mapsto} A(t)\in\Rnn$ of the form
	\begin{equation}
		A(t) = U(t)\Sigma(t)V(t)^\top 
		\label{eq:ambientCurve}
	\end{equation}
	with 
	\begin{equation}
		U(t):=\exp\pa{t\Omega_U}, \quad \Sigma(t):=\exp(t)D,\quad V(t):= \exp(t\Omega_V),
	\end{equation}
	for some $n\times n$ skew-symmetric matrices $\Omega_U,\Omega_V$ and a diagonal matrix $D = \mathrm{diag}(\sigma_1,\dots,\sigma_n)$, for a positive and geometrically decaying sequence $\sigma_i$. A rank-$r$ approximation of this curve is reconstructed by numerically integrating with the given scheme the DLRA equation~\eqref{eq:odeDLRA} where the scalar field $F$ is replaced by the exact derivative of the ambient curve~\eqref{eq:ambientCurve}:
	\begin{align}
		A'(t) &= U(t) \Big(\Omega_U\Sigma(t) + \Sigma(t) + \Sigma(t) \Omega_V^\top \Big) V(t)^\top .
		\label{eq:ambientDerivative}
	\end{align}
	The approximation error at final time is constituted mainly of the integration error which can be reduced by decreasing the step size, and the modeling error is affected only by the choice of $r$. A scheme is said to be robust to small singular values, if the integration error is independent of the choice of $r$. In practice, one must observe that the trend of the error as a function of the step size is unaffected by the choice of $r$ for step sizes where the modeling error is negligible compared to the integration error.
	
	Figure~\ref{fig:stabilityToSmallSingularValues} shows the results for the experiment described in the previous paragraph with a curve of the form~\eqref{eq:ambientCurve} with randomly generated $\Omega_U$ and $\Omega_V$, initial singular values $\sigma_i = 2^{-i}$ and $n=100$. The panels from left to right corresponds respectively to the AFE, the \revision{PRH} and the A\revision{PRH} schemes. Note that for the AFE and the A\revision{PRH} schemes, we use the exact expression for the second derivative of~\eqref{eq:ambientCurve} given by
	\begin{align}
		A''(t) &= U(t) \Big(\Omega_U^2\Sigma(t) + \Sigma(t) + \Sigma(t) (\Omega_V^2)^\top \\ &\quad\quad\quad + 2\Omega_U\Sigma(t) + 2\Omega_U\Sigma(t) \Omega_V^\top  + 2\Sigma(t) \Omega_V^\top   \Big) V(t)^\top .
	\end{align}
	The results for AFE show the ideal outcome: the error curves for increasing values of $r$ are superimposed until the modeling error plateau determined by the value of $r$ is reached. These results empirically suggest that the AFE integration scheme is robust to small singular values. On the other hand, the \revision{PRH} and A\revision{PRH} \revision{schemes} which rely on retraction-based Hermite interpolation suffer from small singular values. Panels (b) and (c) of Figure~\ref{fig:stabilityToSmallSingularValues} exhibit the same oscillatory convergence trend that could be observed for both schemes in the experiments on the differential Lyapunov equation in Section~\ref{ss:lyapunovExperiments}.

	\begin{figure}
		\begin{minipage}{0.325\linewidth}
			\centering
			\centering \includegraphics[trim = 0.1cm 0.2cm 1.15cm 1cm, clip, width=\linewidth]{ 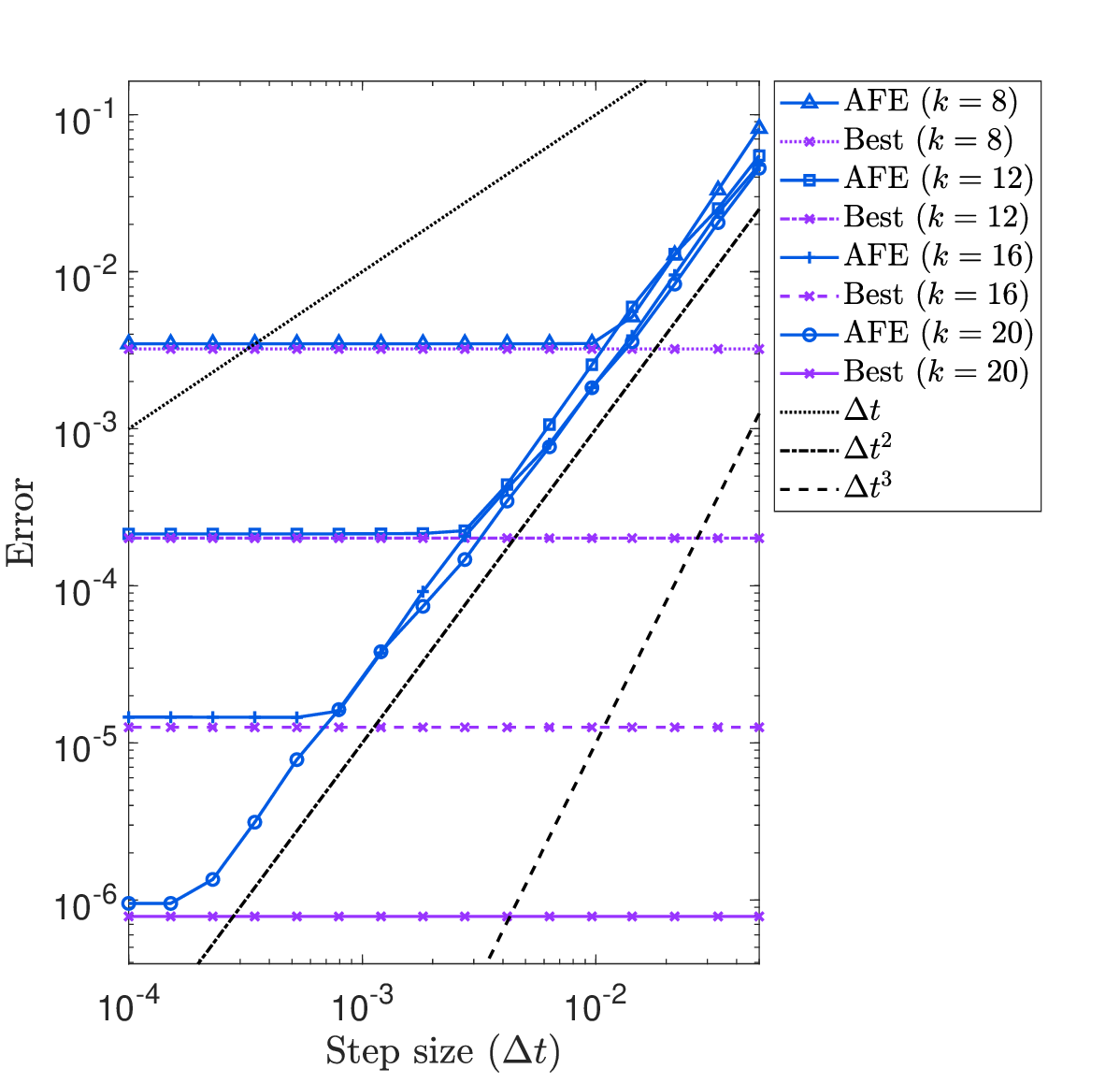}
			\subcaption{AFE scheme}
			\label{fig:robustToSmallSV_AFE}
		\end{minipage}
		\begin{minipage}{0.325\linewidth}
			\centering
			\centering \includegraphics[trim = 0.1cm 0.2cm 1.15cm 1cm, clip,width=\linewidth]{ 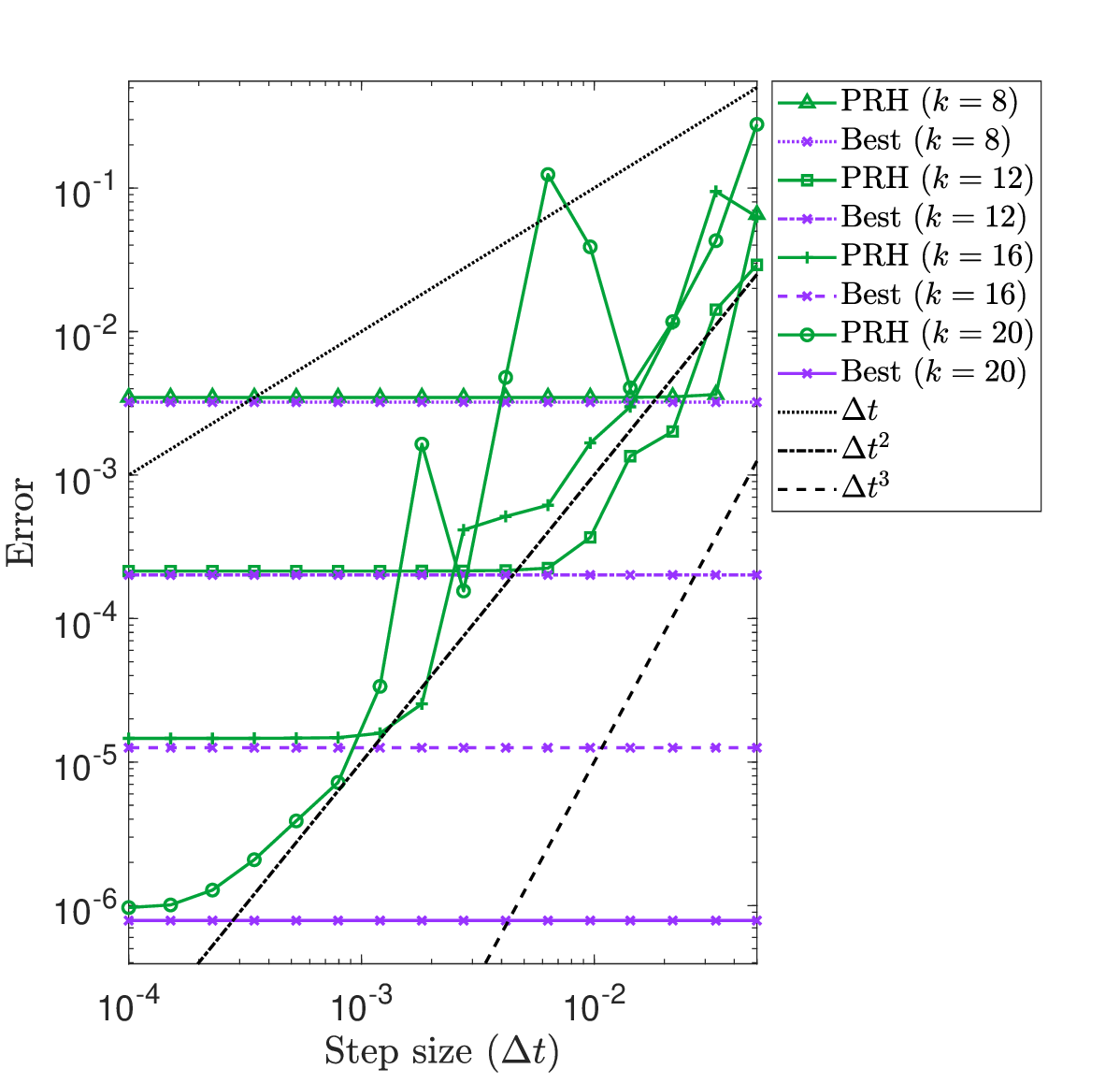}
			\subcaption{\revision{PRH} scheme}
			\label{fig:robustToSmallSV_RH}
		\end{minipage}
		\begin{minipage}{0.325\linewidth}
			\centering
			\centering \includegraphics[trim = 0.1cm 0.2cm 1.15cm 1cm, clip,width=\linewidth]{ 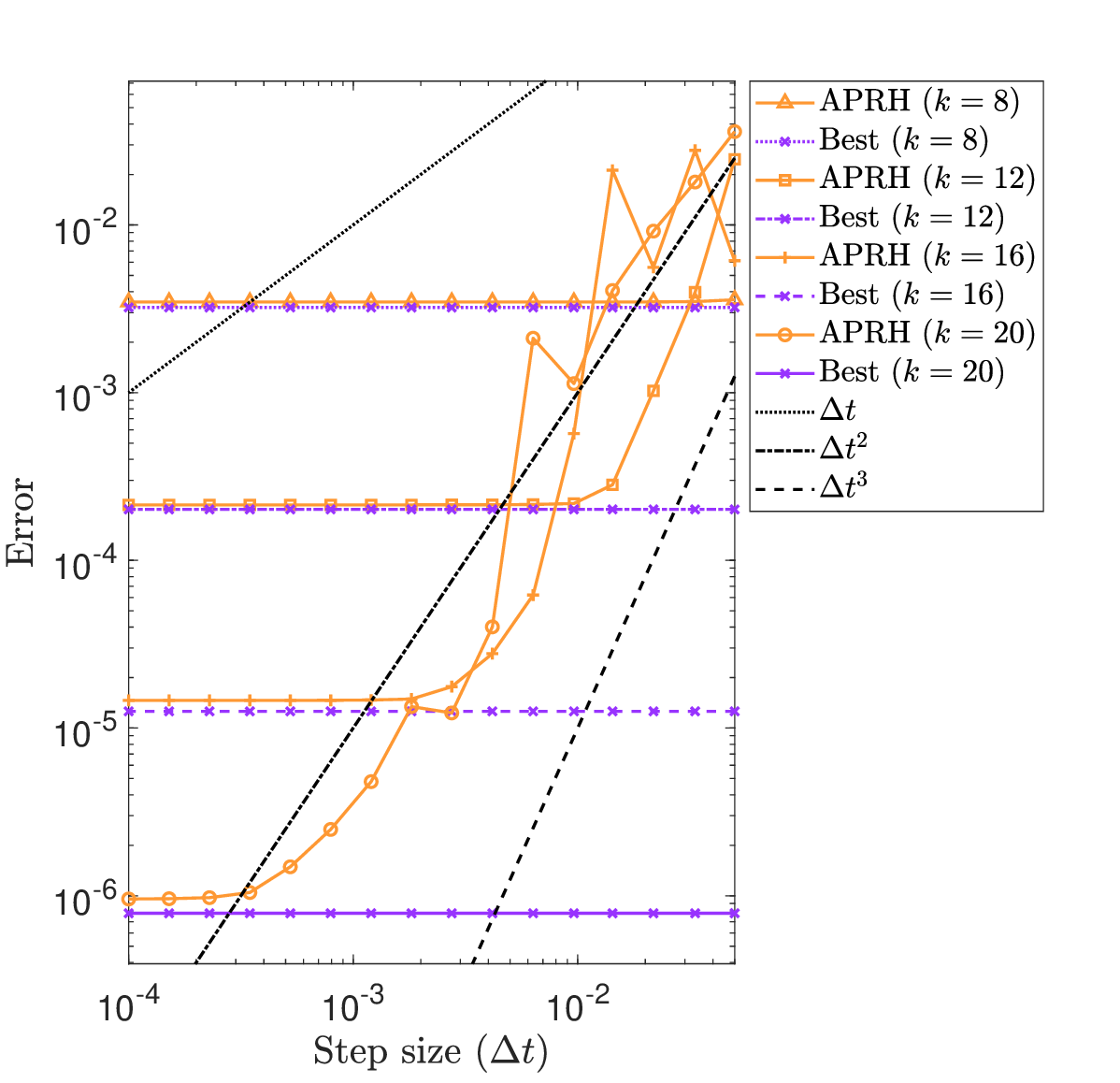}
			\subcaption{\revision{APRH} scheme}
			\label{fig:robustToSmallSV_ARH}
		\end{minipage}
		\caption{Error at final time $\norm{Y_{\Deltat}(T)-A(T)}_2$ versus the step size ${\Deltat}$ of DLRA integration applied to \eqref{eq:ambientDerivative} for reconstructing the curve~\eqref{eq:ambientCurve} for different values of rank.}
		\label{fig:stabilityToSmallSingularValues}
	\end{figure}
	
	A partial explanation for the oscillatory behavior observed in panels (b) and (c) of Figure~\ref{fig:stabilityToSmallSingularValues} for the \revision{PRH} and the A\revision{PRH} schemes comes from studying robustness of the retraction-based Hermite interpolant~\eqref{eq:HermiteInterpolant} to the presence of small singular values at the interpolation points. Consider the following experiment. Take $Y_{0}\left(\sigma_{r}\right) \in \mathcal{M}_{r} \subset \mathbb{R}^{n \times n}$ with $n=100$ , $r=12$ defined by 
	\begin{equation*}
		Y_0(\sigma_r) = U_0 \mathrm{diag}(1,\dots,\sigma_r) V_0^\top 
	\end{equation*}
	for some randomly generated orthogonal matrices $U_0$ and $V_0$ and with $\sigma_i$ logarithmically spaced on the interval $[\sigma_r,1]$, for some $\sigma_r\leq1$. To obtain the second interpolation point, we first move away from $Y_0$ with the orthographic retraction along a random tangent vector $Z\in T_Y\mcal_r$ such that $\norm{Z}_{\mathrm{F}}$ = 1 to get $\tilde Y_1 = \Retr_{Y_0(\sigma_r)}(Z)$. Then, \revision{the} second interpolation point is $Y_1$, obtained from $\tilde Y_1$ by replacing its singular values with
	\begin{equation*}
		\sigma_i(Y_1) = \sigma_i(Y_0)(1 + \xi_i),
	\end{equation*}
	for some random $\xi_i$ drawn from a uniform distribution on $\pac{1/2,2}$. This way, the singular values decay of both $Y_0$ and $Y_1$ mimic a situation encountered in one step of the \revision{PRH} and A\revision{PRH} integration schemes, when the smallest singular value of the current approximation is of the order of $\sigma_r$.  
	Then, we randomly generate $Z_0\in T_{Y_0}\mcal_r$ and $Z_1\in T_{Y_1}\mcal_r$ with $\|Z_0\| = \|Z_1\| = 1$ and form the retraction-based interpolant~\eqref{eq:HermiteInterpolant} given by
	\begin{equation}
		H(\tau) = H(\tau; \,\pa{0,Y_0,Z_0}, \,\pa{1,Y_1,Z_1}),\quad\tau\in\pac{0,1}.
		\label{eq:stressedHermiteInterpolant}
	\end{equation}
	 For different values of the smallest singular value $\sigma_r$, we measure the discrepancies $\normbig{Z_0 - \dot H(0)}_{\mathrm{F}}$ and $\normbig{Z_1 - \dot H(1)}_{\mathrm{F}}$, where derivatives of $H$ are obtained by finite differences. 
	The experiment is repeated for each $\sigma_r$ on $100$ randomly generated instances and the error distribution is plotted against $\sigma_r$ in Figure~\ref{fig:robustnessOfRetractionHermiteToSmallSV}.  These results unequivocally indicate the fragility of retraction-based Hermite interpolant on the fixed-rank manifold when small singular values are present in the interpolation points. As $\sigma_r$ decrease, the velocity error in $\tau = 0$ increases, and even more severely in $\tau = 1$. The fact that the error is non-negligible even for moderately small values of $\sigma_r$ suggests the \revision{PRH} and A\revision{PRH} integration schemes may occasionally employ inaccurate retraction-based Hermite interpolants. This may contribute to the oscillatory behavior of the error observed \revision{PRH} and A\revision{PRH} in Figures~\ref{fig:lyapunovResultsConvergence}~and~\ref{fig:stabilityToSmallSingularValues}.
	\begin{figure}
		\begin{minipage}{0.495\linewidth}
			\centering \includegraphics[width=\linewidth, trim=0.5cm .1cm 1cm .1cm, clip]{ 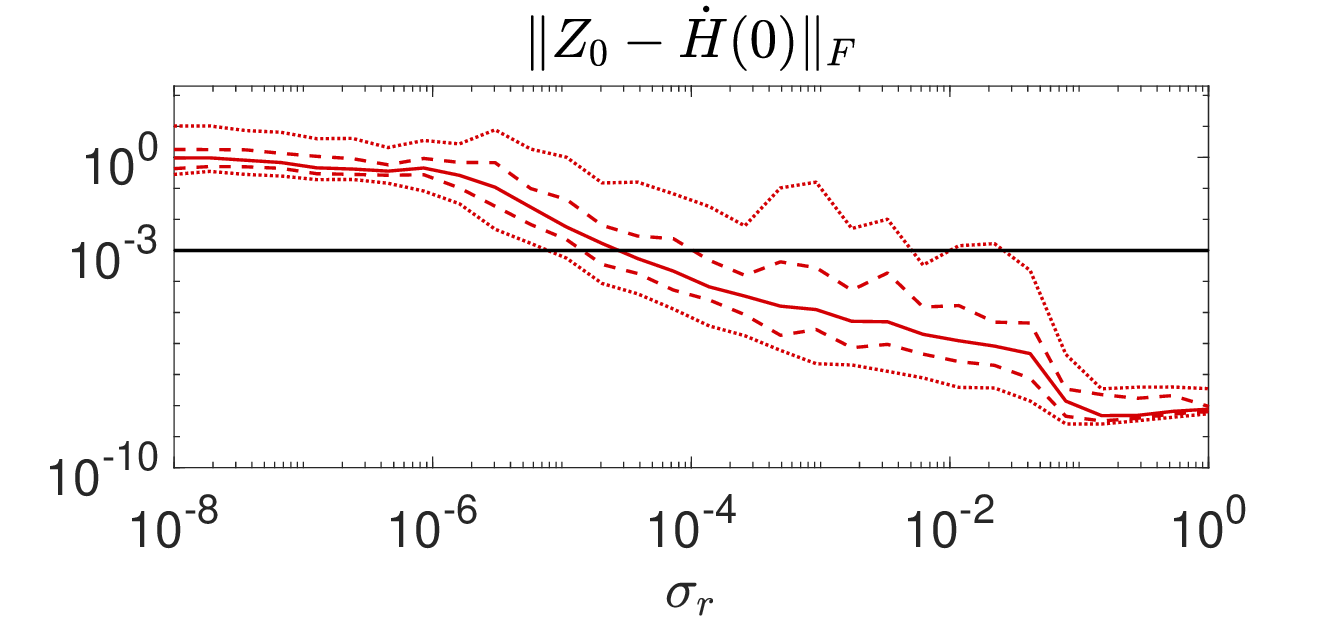}
		\end{minipage}
		\begin{minipage}{0.495\linewidth}
			\centering \includegraphics[width=\linewidth, trim=0.5cm .1cm 1cm .1cm, clip]{ 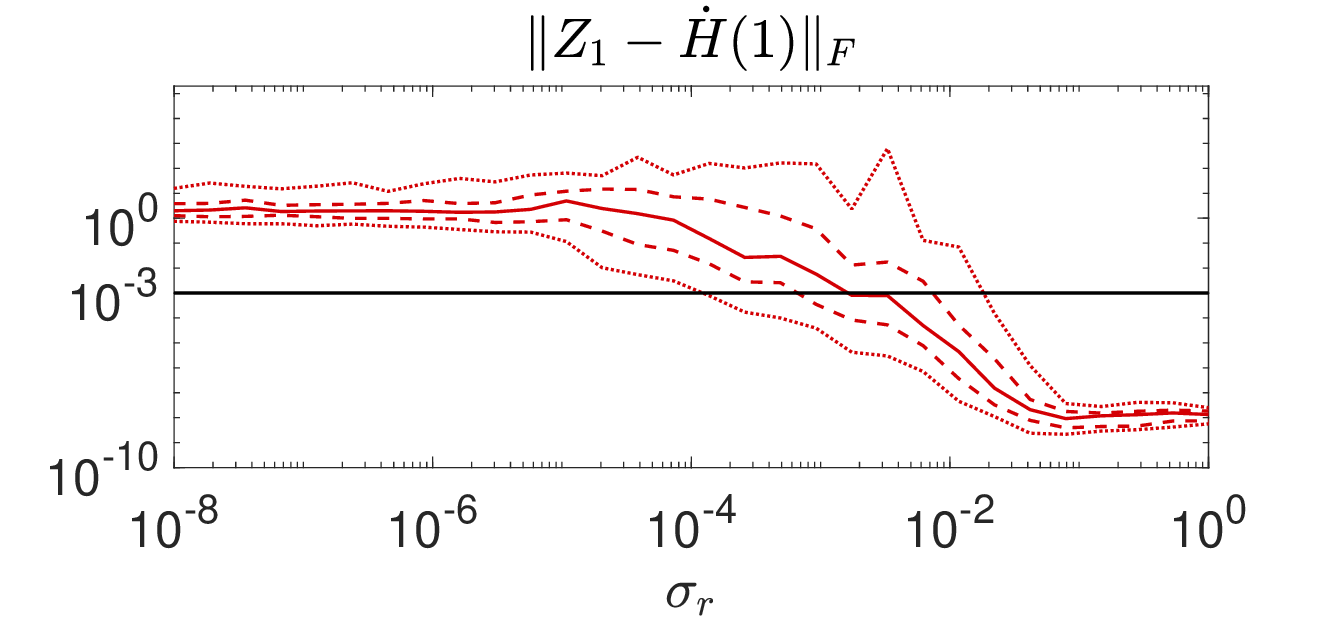}
		\end{minipage}
		\caption{Robustness to small singular values of the retraction-based Hermite interpolant~\eqref{eq:stressedHermiteInterpolant}. The solid line is the median of the error over 100 randomly generated instances for each value of $\sigma_r$ while the dashed and dotted lines correspond to the percentiles $\pac{0.05,0.25,0.75,0.95}$ of the sampled error.}
		\label{fig:robustnessOfRetractionHermiteToSmallSV}
	\end{figure}

	\section{Conclusion}\label{s:conclusion}
	This work contributes to strengthening the connection between retractions and numerical integration methods for manifold ODEs and especially DLRA techniques. In particular, we show that the so-called unconventional integration scheme~\cite{unconventionalIntegrator} defines a second-order retraction which approximates up to high-order terms the orthographic retraction. It remains an open question whether the same observation can be made for the recently proposed parallelized version of KLS~\cite{ceruti2023parallel}.

	 We also derive three numerical integration schemes expressed in terms of retractions and showcase their performance on classic problem instances of DLRA. The derivation and the numerical results show that the methods can achieve second-order error convergence with respect to the time integration step. However, the methods have shown mixed results. 
	 While the AFE and the A\revision{PRH} schemes exhibit instability in the presence of large normal components of the ambient vector field, the \revision{PRH} scheme appears more resilient to this aspect. On the other hand, the occurrence of small singular values in the approximation had no apparent effect on the  performance of AFE but for the \revision{PRH} and A\revision{PRH} methods, small singular values may explain occasional deviations from the expected second-order convergence behavior. \revision{We observe that the PRH scheme delivers similar performance, both with respect to computational time and accuracy, compared to its existing counterpart, PRK2. However, the high-order accelerated version, APRH, is found to be less favorable compared to analogous schemes such as PRK3. This is largely due to the additional computational cost incurred by the Weingarten map.}
		
	For other low-rank tensor formats, such as the Tucker or the tensor-train formats, retractions have also been proposed~\cite{steinlechnerTucker,steinlechnerTT}. However, to the best of our knowledge no retraction with an efficiently computable inverse retraction is known and the orthographic retraction has remained elusive due to the complexity of the normal space structure for these manifolds. Yet, the KLS scheme has been extended to low-rank tensor manifolds~\cite[\S 5]{unconventionalIntegrator}. Hence, assuming the connection with the orthographic retraction carries over to the tensor setting, it may be possible to retrieve the orthographic retraction for such tensor manifolds as a small perturbation of the KLS update. Then, the possibility to easily compute the inverse orthographic retraction would enable using also in the case of low-rank tensor manifolds the retraction-based endpoint curves and the numerical integration schemes presented in this work. 

	\bibliography{bibliography}

\end{document}